\documentclass[reqno]{amsart}
\usepackage[utf8x]{inputenc}  

\usepackage{setspace}
\usepackage[left=3.1cm, right=3.1cm, bottom=4cm]{geometry}                 

\usepackage{graphicx} 
\usepackage{pgf,tikz}
\usetikzlibrary{arrows}
\usepackage{amssymb}
\usepackage{pdfsync}
\usepackage{mathrsfs}
\usepackage{hyperref} 
\usepackage{verbatim} 
\usepackage{epstopdf}
\usepackage{bbm} 
\usepackage[colorinlistoftodos,prependcaption,textsize=tiny]{todonotes}
\usepackage{xargs}
\usepackage{bm}

\def\R{\mathbb{R}}

\def\N{\mathbb{N}}
\def\Z{\mathbb{Z}}
\def\C{\mathbb{C}}
\def\H{\mathbb{H}}

\def\A{\mathcal{A}}

\def\supp{{\rm supp}}
\def\H{\mathcal{H}}

\renewcommand{\d}{\text{\rm d}}

\def\sgn {{\rm sgn}}
\newcommand{\mc}{\mathcal}

\newtheorem{theorem}{Theorem}

\newtheorem{corollary}[theorem]{Corollary}

\newtheorem*{definition*}{Definition}

\newtheorem{proposition}[theorem]{Proposition}
\newtheorem{lemma}[theorem]{Lemma}

\makeatletter
\DeclareFontFamily{U}{tipa}{}
\DeclareFontShape{U}{tipa}{m}{n}{<->tipa10}{}
\newcommand{\arc@char}{{\usefont{U}{tipa}{m}{n}\symbol{62}}}%

\makeatother

\numberwithin{equation}{section}

\allowdisplaybreaks

\newcommand{\intav}[1]{\mathchoice {\mathop{\vrule width 6pt height 3 pt depth  -2.5pt
\kern -8pt \intop}\nolimits_{\kern -6pt#1}} {\mathop{\vrule width
5pt height 3  pt depth -2.6pt \kern -6pt \intop}\nolimits_{#1}}
{\mathop{\vrule width 5pt height 3 pt depth -2.6pt \kern -6pt
\intop}\nolimits_{#1}} {\mathop{\vrule width 5pt height 3 pt depth
-2.6pt \kern -6pt \intop}\nolimits_{#1}}}

\newcommand{\intavl}[1]{\mathchoice {\mathop{\vrule width 6pt height 3 pt depth  -2.5pt
\kern -8pt \intop}\limits_{\kern -6pt#1}} {\mathop{\vrule width 5pt
height 3  pt depth -2.6pt \kern -6pt \intop}\nolimits_{#1}}
{\mathop{\vrule width 5pt height 3 pt depth -2.6pt \kern -6pt
\intop}\nolimits_{#1}} {\mathop{\vrule width 5pt height 3 pt depth
-2.6pt \kern -6pt \intop}\nolimits_{#1}}}

  \newcommand{\z}{{\bm z}}
    \newcommand{\x}{{\bm x}}
      \newcommand{\y}{{\bm y}}
        \renewcommand{\t}{{\bm t}}
         
    \newcommand{\xxi}{{\bm \xi}}   
    \renewcommand{\j}{{\frak j}}   

 \newcommand{\ov}{\overline}

\title[Sharp embeddings between weighted Paley-Wiener spaces]{Sharp embeddings between weighted Paley-Wiener spaces}

\author[Carneiro]{Emanuel Carneiro}
\author[Gonz\'{a}lez-Riquelme]{Cristian Gonz\'{a}lez-Riquelme}
\author[Oliveira]{Lucas Oliveira}
\author[Olivo]{Andrea Olivo}
\author[Ombrosi]{Sheldy Ombrosi}
\author[Ramos]{Antonio Pedro Ramos}
\author[Sousa]{Mateus Sousa}

\address[Emanuel Carneiro]{
ICTP - The Abdus Salam International Centre for Theoretical Physics, 
Strada Costiera, 11, I - 34151, Trieste, Italy.}
\email{carneiro@ictp.it}

\address[Cristian Gonz\'{a}lez-Riquelme]{Centre de Recerca Matem\`{a}tica, Campus de Bellaterra, Edifici C 08193 Bellaterra (Barcelona),
Spain.}
\email{cgonzalez@crm.cat}

\address[Lucas Oliveira]{Departamento de Matem\'{a}tica Pura e Aplicada, Universidade Federal do Rio Grande do Sul, Porto Alegre, RS, Brazil 91509-900.}
\email{lucas.oliveira@ufrgs.br}

\address[Andrea Olivo]{Departamento de Matem\'{a}tica, Facultad de Ciencias Exactas y Naturales, Universidad de Buenos Aires and IMAS-CONICET, Pabell\'{o}n I (C1428EGA), Ciudad de Buenos Aires, Argentina.}
\email{aolivo@dm.uba.ar}

\address[Sheldy Ombrosi]{Departamento de Matemática e Instituto de Matemática
de Bahía Blanca. Universidad Nacional del Sur - CONICET Bahía Blanca, Argentina.}
\email{sombrosi@uns.edu.ar}

\address[Antonio Pedro Ramos]{IMPA - Instituto de Matem\'atica Pura e Aplicada, Rio de Janeiro - RJ, Brazil, 22460-320.}
\email{antonio.ramos@impa.br}

\address[Mateus Sousa]{BCAM - Basque Center for Applied Mathematics, Alameda de Mazarredo 14, 48009 Bilbao, Bizkaia, Spain.}
\email{mcosta@bcamath.org}

\date{\today}                                           
\begin{document}

\begin{abstract} In this paper we address the problem of estimating the operator norm of the embeddings between multidimensional weighted Paley-Wiener spaces. These can be equivalently thought as Fourier uncertainty principles for bandlimited functions. By means of radial symmetrization mechanisms, we show that such problems can all be shifted to dimension one. We provide precise asymptotics in the general case and, in some particular situations, we are able to identify the sharp constants and characterize the extremizers. The sharp constant study is actually a consequence of a more general result we prove in the setup of de Branges spaces of entire functions, addressing the operator given by multiplication by $z^k$, $k \in \N$. Applications to sharp higher order Poincar\'{e} inequalities and other related extremal problems are discussed.

\end{abstract}

\maketitle 
\setcounter{tocdepth}{1}
\tableofcontents

\section{Introduction}

\subsection{Prelude} This paper touches on a few themes within analysis and approximation theory. Our normalization for the Fourier transform of a function $F\in L^1(\R^d)$ is  
\begin{equation*}
\widehat{F}(\xxi)=\int_{\R^d}e^{-2\pi i \x \cdot \xxi}\, F(\x)\,\d \x.
\end{equation*}
As a sample of our study, we start by presenting the following sharp inequality for bandlimited functions (i.e. functions with compactly supported Fourier transforms), which is a particular case of our more general results in Theorems \ref{Teoremacco_versao_dB_Intro} and \ref{Thmaco_class} below.
\begin{corollary}\label{Cor_sharp_1}
Let $f \in L^2(\R)$ with $\supp\big(\widehat{f}\,\big) \subset [-\frac12, \frac12]$. Then
\begin{align}\label{20230218_17:49}
\int_{-\infty}^{\infty} |f(x)|^2 \,\d x \leq \int_{-\infty}^{\infty} |f(x)|^2 \,  x^6\, \d x.
\end{align} 
This inequality is sharp and the unique extremizer (up to multiplication by a complex constant) is 
\begin{align*}
f(x) =   
\frac{(x^2-1)\coth\big(\tfrac{\pi \sqrt{3}}{2}\big) \cos \pi x  - \sqrt{3}\,x \sin \pi x}{x^6-1}.
\end{align*}
\end{corollary}

We are inspired here by the classical work of Holt and Vaaler \cite{HV}, that approaches certain extremal problems in harmonic and complex analysis via the powerful theory of de Branges spaces of entire functions \cite{Branges}.  Our aim is to study certain sharp inequalities for bandlimited functions that fall under the paradigm of {\it Fourier uncertainty}, i.e. that one cannot have an unrestricted control of a function and its Fourier transform simultaneously. Different manifestations of Fourier uncertainty are ubiquitous in harmonic analysis; for a brief historical account and many important examples we refer the reader to the survey articles \cite{{BDsurvey}, FSsurvey} and the references therein. Other recent works that apply de Branges spaces techniques in connection to extremal problems in analysis include \cite{BCOS, CCLM, CarChiMil, CarGon, CL1, CL2, L, LS1, LS2, V2}. 

\smallskip

As we shall see, some of the inequalities discussed in this paper can be equivalently formulated, via the Fourier transform, in terms of classical derivatives (i.e. {\it Poincaré inequalities}). When such a reformulation is available, our approach via de Branges spaces provides an alternative framework to look into the problem of finding sharp constants, in contrast to the classical Sturm-Liouville approach on the Fourier side; more details on \S \ref{Shapr_poincare_section} below. 

\smallskip

The main results in this paper are the ones described in this introduction together with the general Theorems \ref{Teoremacco_versao_dB} and \ref{Thmaco_class}, presented in Section \ref{Sec8_mult_z}.

\subsection{Setup} \label{Sec_Setup} We start with our general setup, in the language of entire functions of several complex variables. Vectors in $\R^d$ or $\C^d$ are denoted here with bold font (e.g. $\x$, $\y$, $\z$) and numbers in $\R$ or $\C$ with regular font (e.g. $x,y,z$). For $\z = (z_1, z_2, \ldots, z_d) \in \C^d$ we let $|\cdot|$ be the usual norm $|\z| := (|z_1|^2 + \ldots + |z_d|^2)^{1/2}$, and define a second norm $\|\cdot\|$ by 
$$\|\z\| := \sup \left\{ \left|\sum_{n=1}^d z_n\,t_n\right|; \ \t \in \R^d \ {\rm and} \ |\t|\leq 1\right\}.$$
If $F: \C^d \to \C$ is an entire function of $d$ complex variables, which is not identically zero, we say that $F$ has {\it exponential type} if
$$\tau(F):= \limsup_{\|\z\|\to \infty} \|\z\|^{-1}\,\log|F(\z)| < \infty.$$
In this case, the number $\tau(F)$ is called the exponential type of $F$. When $d=1$ this is the classical definition of exponential type and, when $d \geq 2$, our definition is a particular case of a more general concept of exponential type with respect to a compact, convex and symmetric set $K \subset \R^d$ (cf. \cite[pp. 111-112]{SW}). In our case, this convex set $K$ is simply the unit Euclidean ball.

\smallskip

For each $ \alpha > -1$ and $\delta >0$, we let $\mc{H}_{\alpha}(d\,; \delta)$ be the Hilbert space of entire functions $F: \C^d \to \C$ of exponential type at most $\delta$ such that 
\begin{equation}\label{20221121_11:09}
\|F\|_{ \mc{H}_{\alpha}(d; \delta)}:= \left(\int_{\R^d} |F(\x)|^2 \, |\x|^{2\alpha + 2 -d}\,\d\x\right)^{1/2} <\infty.
\end{equation}
For generic $\alpha \geq \beta > -1$ observe that one has the inclusion $\mc{H}_{\alpha}(d\,; \delta) \subset \mc{H}_{\beta}(d\,; \delta)$, and from the closed graph theorem it follows that the inclusion map $I: \mc{H}_{\alpha}(d\,; \delta) \to \mc{H}_{\beta}(d\,; \delta)$ is a bounded operator. The purpose of this paper is to investigate the operator norm $\|I\|_{\mc{H}_{\alpha}(d\,; \delta) \to \mc{H}_{\beta}(d\,; \delta)}$ of this embedding, 
which can be equivalently thought of as the quest to find the sharp form of a Fourier uncertainty principle. In our study, it will be convenient to consider the following formulation: 

\smallskip

\noindent {\it Extremal Problem 1} (EP1): For $\alpha \geq \beta > -1$ and $\delta >0$ real parameters, and $d \in \N$, find the value of
\begin{equation}\label{20220815_17:17_1}
(\mathbb{EP}1)(\alpha, \beta\,; d\,; \delta):= \inf_{ 0 \neq F \in  \mc{H}_{\alpha}(d; \delta) } \frac{\int_{\R^d} |F(\x)|^2 \, |\x|^{2\alpha + 2 -d}\,\d\x}{\int_{\R^d} |F(\x)|^2 \, |\x|^{2\beta + 2 -d}\,\d\x}.
\end{equation}
This is the first of four extremal problems that will appear in this paper, and our choice of notation for the sharp constant aims to facilitate such references. In particular, note that $(\mathbb{EP}1)(\alpha, \beta\,; d\,; \delta) = \big(\|I\|_{\mc{H}_{\alpha}(d\,; \delta) \to \mc{H}_{\beta}(d\,; \delta)}\big)^{-2}$ and hence
\begin{equation}\label{20220816_14:41_1}
(\mathbb{EP}1)(\alpha, \beta\,; d\,; \delta) >0.
\end{equation} 
Inequality \eqref{20220816_14:41_1} can be viewed as a Fourier uncertainty principle. In fact, the Paley-Wiener theorem (see \cite[Chapter III, Section 4]{SW} and \cite[Theorem 1.7.5 and Theorem 1.7.7]{Hor}) tells us that $F \in \mc{H}_{\alpha}(d\,; \delta)$ has (distributional) Fourier transform supported in the closed ball of radius $\delta/2\pi$ centered at the origin and, therefore, the Fourier uncertainty paradigm implies that the mass of $F$ cannot be too concentrated around the origin. 

\smallskip

We now move to describing our main results. We split them into three distinct parts.

\subsection{Main results}

\subsubsection{Part I: Qualitative properties} \label{Part I: symmetries and extremizers} Observe that $F(\cdot) \in \mc{H}_{\alpha}(d\,; \delta)$ if and only if $F(\cdot/ \delta)\in \mc{H}_{\alpha}(d\,; 1)$. This change of variables plainly yields the relation
\begin{equation}\label{20220816_15:18_1}
(\mathbb{EP}1)(\alpha, \beta\,; d\,; \delta) = \delta^{2 \beta - 2\alpha}\, (\mathbb{EP}1)(\alpha, \beta\, ; d\, ;1).
\end{equation}
Our first non-trivial observation about the extremal problem (EP1) is as follows.
\begin{theorem} [Dimension shifts]\label{Thm_dimension_shifts} We have 
\begin{equation}\label{20220816_15:19_1}
(\mathbb{EP}1)(\alpha, \beta\, ; d\,; \delta) = (\mathbb{EP}1)(\alpha, \beta\,; 1\,;\delta).
\end{equation}
\end{theorem}
The proof of Theorem \ref{Thm_dimension_shifts} is carried out in Section \ref{Sec_proof_Thm1_Dim_shifts} and it relies on suitable radial symmetrization mechanisms and an auxiliary extremal problem that may be of independent interest. In view of \eqref{20220816_15:18_1} and \eqref{20220816_15:19_1}, when it comes to the extremal problem (EP1), we can therefore restrict our attention to dimension $d=1$ with any particular exponential type $\delta$ of our choice. 

\smallskip

A function $0 \neq F \in  \mc{H}_{\alpha}(d\,; \delta)$ is said to be an {\it extremizer} for $(\mathbb{EP}1)(\alpha, \beta\,; d\,; \delta)$ if it realizes the infimum in \eqref{20220815_17:17_1}, i.e. if
$$\int_{\R^d} |F(\x)|^2 \, |\x|^{2\alpha + 2 -d}\,\d\x = (\mathbb{EP}1)(\alpha, \beta\,; d\,; \delta)\int_{\R^d} |F(\x)|^2 \, |\x|^{2\beta + 2 -d}\,\d\x.$$ 
We say that an entire function $F:\C^d \to \C$ is radial, if its restriction to $\R^d$ is radial. Our next result addresses the existence of extremizers.
\begin{theorem}[Radial extremizers] \label{Thm2}
There exists a radial extremizer for $(\mathbb{EP}1)(\alpha, \beta\,; d\,; \delta)$.
\end{theorem}

We discuss the proof of this result in Section \ref{Sec_Existence_of _extremizers}. In dimension $d=1$ we go further and establish that any extremizer of $(\mathbb{EP}1)(\alpha, \beta\,; 1\,; \delta)$ must be an even function; see Proposition \ref{Prop_even_extremizers_dim1}.

\smallskip

Our next result is concerned with the continuity of $(\mathbb{EP}1)(\alpha, \beta\,; d\,; \delta)$ in all real-valued parameters.

\begin{theorem}[Continuity] \label{Thm_continuity}
The function $(\alpha, \beta, \delta) \mapsto (\mathbb{EP}1)(\alpha, \beta\,; d\,; \delta)$ is continuous in the range $\alpha \geq \beta >-1$ and $\delta >0$. 
\end{theorem}

The proof of this result is carried out in Section \ref{Sec_cont_proof}.

\subsubsection{Part II: Asymptotics} Finding the exact value of the constant $(\mathbb{EP}1)(\alpha, \beta\, ; d\,; \delta)$ is not a simple task in general, and one is naturally led to study this object from an asymptotic point of view. As observed in \S \ref{Part I: symmetries and extremizers}, from now on let us restrict ourselves to the scenario $d=1$ and $\delta=1$, without loss of generality. That is, we are looking at \footnote{The reader may wonder why we choose the normalization $|x|^{2\alpha +1}$ instead of $|x|^{2\alpha}$ (which is perhaps more common in PDEs given that it would be related to the derivative or order $\alpha$). The reason is that, as we shall see, these problems are related to the classical Bessel functions of the first kind $J_{\alpha}$, and we opted for the simplicity of notation at this end.
}
\begin{equation*}
(\mathbb{EP}1)(\alpha, \beta\,; 1\,; 1):= \inf_{ 0 \neq f \in  \mc{H}_{\alpha}(1; 1) } \frac{\int_{\R} |f(x)|^2 \, |x|^{2\alpha +1}\,\d x}{\int_{\R} |f(x)|^2 \, |x|^{2\beta +1}\,\d x}.
\end{equation*}
It is convenient to introduce a logarithmic scale, and look at the function $(\alpha, \beta) \mapsto \log\big( (\mathbb{EP}1)(\alpha, \beta\,; 1\,; 1)\big)$. We aim to find a description of the form 
$$\log\big( (\mathbb{EP}1)(\alpha, \beta\,; 1\,; 1)\big) = M(\alpha, \beta)  + R(\alpha, \beta)\,,$$
where $M$ would be a main term as $\alpha, \beta$ move towards the limits $-1$ and/or $\infty$, and $R$ would be a remainder term. Also, when $\alpha = \beta$ we know that $\log\big( (\mathbb{EP}1)(\alpha, \beta\,; 1\,; 1)\big) = 0$, and we would like our $M$ and $R$ to reflect that, possibly under some reasonable control. This is the spirit of our next result.

\begin{theorem}[Asymptotics]\label{Thm5 - Asympt} For $\alpha \geq \beta > -1$ we have
\begin{align}\label{20230131_17:29}
\begin{split}
\log\big( (\mathbb{EP}1)(\alpha, \beta\,; 1\,; 1)\big)  = 2(\alpha& - \beta)\log(\alpha +2)  + \log \left(\frac{\beta +1}{\alpha +1}\right) \\
&  + O\left( \left(\frac{(\alpha - \beta)(\alpha +2)}{(\alpha +1)}\right) \, \log \left( \frac {2(\alpha +1)(\alpha - \beta + 1)}{(\alpha - \beta)(\alpha +2)}\right)\right),
\end{split}
\end{align}
where the implied constant is universal.
\end{theorem}
We described our error term above in a unified expression for brevity, but a careful look at it reveals its behaviour in the different regimes. Let us briefly reflect on the possibilities. Writing $x = \alpha - \beta$ and $y = \alpha +1$ for simplicity, note that $0 \leq x <y$ and $\frac{y(x+1)}{x(y+1)} \geq 1$. If $\alpha$ is away from $-1$, say $\alpha \geq -\frac12$, then our error term is $O\big(x \log (\frac{2(x+1)}{x})\big)$, and this can be further understood as $O\big(x \log (\frac{1}{x})\big)$ for $x$ small and $O(x)$ for $x$ large. If $\alpha \leq -\frac12$, our error term is $O\big(\frac{x}{y}\log (\frac{2y}{x})\big)$ which, in particular, is bounded. If $\alpha \geq \beta \geq -\frac12$, since $\log \left(\frac{\beta +1}{\alpha +1}\right) = O(\alpha - \beta)$ (just observe that the derivative of $s \mapsto \log(s+1)$ is bounded in this range and apply the mean value theorem), our expression becomes
$$\log\big( (\mathbb{EP}1)(\alpha, \beta\,; 1\,; 1)\big)  = 2(\alpha - \beta)\log(\alpha +2) + O\left((\alpha - \beta) \log \left(\frac{2(\alpha - \beta+1)}{\alpha - \beta}\right)\right).$$
Similarly, if $-1 < \beta \leq \alpha \leq -\frac12$, we have 
$$\log\big( (\mathbb{EP}1)(\alpha, \beta\,; 1\,; 1)\big)  = \log \left(\frac{\beta +1}{\alpha +1}\right) + O\left(\left(\frac{\alpha - \beta}{\alpha +1} \right) \log \left(\frac{2(\alpha +1)}{\alpha - \beta}\right)\right).$$
If $\alpha \to \infty$ and $\beta \to -1$ simultaneously, there might be a legitimate match between the two components of the main term in \eqref{20230131_17:29}.

\smallskip

The proof of Theorem \ref{Thm5 - Asympt} is carried out in Section \ref{Sec5_asym_ppp}. For this proof we establish upper and lower bounds that match the proposed asymptotic \eqref{20230131_17:29}. The upper bound is obtained via a suitable example involving the reproducing kernel of the space. The lower bound is subtler, and we need two different mechanisms to treat the two different ranges: $\alpha$ and $\beta$ near the upper endpoint $\infty$, and $\alpha$ and $\beta$ near the lower endpoint $-1$. For the former we rely on quantitative versions of certain classical Fourier uncertainty principles, and run an optimization procedure to arrive at the desired bound. For the latter, where the classical Fourier transform is not available, we use a different strategy, relying on certain estimates involving the reproducing kernel of the space, and again running an optimization procedure on the parameters. It is interesting to notice that, at least as far as our setup goes, neither of these strategies would be sufficient to tackle the whole range on its own. The remaining cases, where $\alpha$ is near $\infty$ and $\beta$ is near $-1$, are addressed via a comparison with a suitable test point between $\beta$ and $\alpha$.

\subsubsection{Part III: Sharp constants} \label{Part3 - SC}We continue to restrict our attention to the scenario $d=1$ and $\delta=1$, without loss of generality. In some special occasions we are able to identify the precise value of the constant $(\mathbb{EP}1)(\alpha, \beta\, ; 1\,; 1)$. Before proceeding to this description, let us fix some notation and review some facts from the theory of Bessel functions that shall be relevant for our purposes.

\smallskip

For $\nu > -1$ let $A_{\nu}:\C \to \C$ and $B_{\nu}:\C \to \C$ be the real entire functions defined by 
\begin{equation}\label{20220816_17:43}
A_{\nu}(z) := \sum_{n=0}^{\infty} \frac{(-1)^n \big(\tfrac12 z\big)^{2n}}{n!(\nu +1)(\nu +2)\ldots(\nu+n)}
\end{equation}
and
\begin{equation}\label{20221121_10:21}
B_{\nu}(z) := \sum_{n=0}^{\infty} \frac{(-1)^n \big(\tfrac12 z\big)^{2n+1}}{n!(\nu +1)(\nu +2)\ldots(\nu+n+1)}.
\end{equation}
These functions are related to the classical Bessel functions of the first kind by the identities
\begin{align}\label{20230210_14:02}
\begin{split}
A_{\nu}(z) &= \Gamma(\nu +1) \left(\tfrac12 z\right)^{-\nu} J_{\nu}(z), \\
B_{\nu}(z) & = \Gamma(\nu +1) \left(\tfrac12 z\right)^{-\nu} J_{\nu+1}(z).
\end{split}
\end{align}
Both $A_{\nu}$ and $B_{\nu}$ have only real, simple zeros and have no common zeros (note that, in the simplest case $\nu = -1/2$, we have $A_{-1/2}(z) = \cos z$ and $B_{-1/2}(z) = \sin z$). The function $A_{\nu}$ is even while $B_{\nu}$ is odd, and we have $\tau(A_{\nu}) =\tau(B_{\nu}) = 1$. Moreover, they satisfy the system of differential equations 
\begin{equation}\label{20221130_14:53}
A_{\nu}'(z) = - B_{\nu}(z) \ \ \ {\rm and} \ \ \ B_{\nu}'(z) = A_{\nu}(z) - (2\nu +1)z^{-1}B_{\nu}(z).
\end{equation}

Our main result in this section is a complete solution for the extremal problem (EP1) in the cases where $\alpha = \beta + k$ for $k \in \N$. The answer is given in terms of the smallest positive solution of a certain explicit determinant equation involving Bessel functions. Let 
$$0 < {\frak j}_{\nu, 1} < \j_{\nu, 2} <  \j_{\nu, 3} < \ldots$$
denote the sequence of positive zeros of the Bessel function $J_{\nu}$, and define the meromorphic function
$$C_{\nu}(z) := \frac{B_{\nu}(z)}{A_{\nu}(z)}.$$
When $\nu = -1/2$, we simply have $C_{-1/2}(z) = \tan z$.

\begin{theorem}[Sharp constants]\label{Teoremacco_versao_dB_Intro}
Let $\beta > -1$, let $k \in \N$ and set $ \lambda_0:=\big((\mathbb{EP}1)(\beta+k, \beta\,; 1\,; 1)\big)^{1/2k}$. 
 \begin{itemize}
 \item[(i)] If $k=1$ we have $\lambda_0 = {\frak j}_{\beta, 1}$. 
 \smallskip
 \item[(ii)] If $k \geq 2$, set $\ell := \lfloor k/2 \rfloor$. Then $\lambda_0$ is the smallest positive solution of the equation 
 \begin{align*}
 A_{\beta}(\lambda) \det \mc{V_{\beta}}(\lambda) = 0\,,
 \end{align*}
 where $\mc{V}_{\beta}(\lambda)$ is the $\ell \times \ell$ matrix with entries
\begin{align*}
\ \ \ \ \ \ \  \ \ \ \ \ \ \ \  \ \ \ \big(\mc{V}_{\beta}(\lambda)\big)_{mj} = \sum_{r=0}^{k-1} \omega^{r (4\ell - 2m - 2j +3)} \,C_{\beta}\big(\omega^{r} \lambda\big)  \ \ \  \ \  \ \ \ (1\leq m,j \leq \ell)\,,
\end{align*}
and $\omega := e^{\pi i/k}$.
 \end{itemize}
\end{theorem}

 \noindent {\sc Remark:} For $k \geq 2$, we shall verify in our proof that $0 < \lambda_0 < \frak{j}_{\beta, \ell +1}$. The function $\lambda \mapsto \det \mc{V}_{\beta}(\lambda)$ is, in principle, a meromorphic function of the variable $\lambda$ that is real-valued on $\R$, but we show that $\lambda \mapsto A_{\beta}(\lambda) \det \mc{V}_{\beta}(\lambda)$ is in fact continuous on the interval $(0,\frak{j}_{\beta, \ell +1})$.

\medskip

We also fully classify the extremizers for $(\mathbb{EP}1)(\beta+k, \beta\,; 1\,; 1)$ when $k \in \N$. When $k=1$ they are the complex multiples of $f(z) = A_{\beta}(z)/(z^2 - \frak{j}_{\beta, 1}^2)$ and, when $k \geq 2$, they have the form \begin{equation*}
f(z) = \sum_{n=1}^{\infty} a_n \frac{ \frak{j}_{\beta,n} \,A_{\beta}(z)}{(z^2 - \frak{j}_{\beta,n}^2)}\,,
\end{equation*}
where $(a_1, a_2, \ldots, a_\ell)$ belongs to the kernel of a certain $\ell \times \ell$ matrix (here $\ell  = \lfloor k/2 \rfloor$) and each $a_n$, for $n >\ell$, is given in terms of $a_1, \ldots, a_{\ell}$; see Theorem \ref{Thmaco_class} for details.

\smallskip

The novelty in our sharp constant approach here is to connect the extremal problem (EP1) with the rich theory of de Branges spaces of entire functions. As a matter of fact, Theorem \ref{Teoremacco_versao_dB_Intro} ends up being a special case of the much more general Theorem \ref{Teoremacco_versao_dB}, that we prove in Section \ref{Sec8_mult_z}. Theorems \ref{Teoremacco_versao_dB} and \ref{Thmaco_class} address the extremal problem (EP4), related to the operator given by multiplication by $z^k$ in a general de Branges space, and are certainly among the key results of this paper as well. We have opted to postpone this discussion to Section \ref{Sec8_mult_z} in order to minimize the technical considerations in this introduction (we shall need a brief overview of the de Branges space theory, which we do in \S \ref{Sub_DB_spaces}), but the reader that is familiar with the theory may take a direct look at this section, which is of independent interest.

\smallskip

An interesting feature of Theorem \ref{Teoremacco_versao_dB_Intro} is that we exploit the fact that extremizers are even functions (see Proposition \ref{Prop_even_extremizers_dim1}) to arrive at a determinant of order $\ell = \lfloor k/2 \rfloor$; this suits well the treatment for small $k$. As an illustration, for $k=2$, $\big((\mathbb{EP}1)(\beta+2, \beta\,; 1\,; 1)\big)^{1/4}$ is the first positive root of $\lambda \mapsto A_{\beta}(\lambda)\big( C_{\beta}(\lambda) - i C_{\beta}(i\lambda)\big)$,
while, for $k=3$, $\big((\mathbb{EP}1)(\beta+3, \beta\,; 1\,; 1)\big)^{1/6}$ is the first positive root of 
\begin{equation}\label{20230320_12:41}
\lambda \mapsto A_{\beta}(\lambda)\big( C_{\beta}(\lambda) - C_{\beta}(\omega\lambda) + C_{\beta}(\omega^2\lambda)\big)\,,\end{equation}
where $\omega = e^{\pi i /3}$. For example, when $\beta = -1/2$, the first few values are:
\begin{center}
        \begin{tabular}{ |c | c | c |c|c|c|c|c|c|}
        \hline
             & $k=0$ & $k=1$ &$k=2$ & $k=3$ & $k=4$ & $k=5$ & $k=6$ & $k=7$ \\
            \hline
            $\big((\mathbb{EP}1)(-\frac12 +k, -\frac12\,; 1\,; 1)\big)^{1/2k}$ & 1 & $\pi/2$ &$2.36\ldots$ & $\pi$ &$3.90\ldots$ &$4.67\ldots$&$5.43\ldots$&$6.18\ldots$\\
      
        \hline
        \end{tabular}
 \end{center} 
       
 \smallskip
 
\noindent With a quick look at the table above, the fact that $\big((\mathbb{EP}1)(\frac52, -\frac12\,; 1\,; 1)\big)^{1/6} = \pi$ stands out as rather curious. This is due to the fact that, when $\beta = -1/2$ and $k=3$, the function in \eqref{20230320_12:41}, reduces to 
$$\lambda \mapsto -\frac{\sin \lambda\, \big( \cos\lambda - \cosh(\sqrt{3}\lambda)\big)}{\cos \lambda + \cosh(\sqrt{3}\lambda)}.$$
With the dilation relation \eqref{20220816_15:18_1} in mind, one has $(\mathbb{EP}1)(\frac52, -\frac12\,; 1\,; \pi)= 1$, which is equivalent to the sharp inequality \eqref{20230218_17:49} presented in Corollary \ref{Cor_sharp_1}. The characterization of the extremizer of \eqref{20230218_17:49} is a consequence of Theorem \ref{Thmaco_class} below, where we find the alternative series representation
\begin{align*}
f(x) = \frac{2\sqrt{3}}{\pi}\sum_{n\geq1} \frac{\left(n - \frac12\right)^2}{\left(\left(n - \frac12\right)^6 - 1\right) } \frac{ \cos{\pi x}}{ \left(x^2 - \left(n - \frac12\right)^2\right)}.
\end{align*}

\subsection{Application 1: sharp higher order Poincar\'{e} inequalities} \label{Shapr_poincare_section} Some particular cases of Theorem \ref{Teoremacco_versao_dB_Intro} can be equivalently formulated in the framework of sharp Poincaré inequalities via the Fourier transform. For $m \in \Z_{\geq0}$, let $W^{m,2}_0(B_r)$ be the usual $L^2$-Sobolev space of order $m$ and zero boundary data, i.e. the closure of $C^{\infty}_c(B_r)$ in $W^{m,2}(B_r)$, where $B_r \subset \R^d$ is the ball of radius $r$ centered at ${\bf 0}$.  

\begin{corollary}\label{20230224_13:44}
Let $m, n \in \Z_{\geq0}$, with $m \geq n$. For any $g \in W^{m,2}_0(-r, r)$ we have
\begin{align}\label{20230223_18:13}
\int_{-r}^{r} |g^{(n)}(x)|^2 \,\d x \leq \frac{r^{2(m-n)}}{(\mathbb{EP}1)(m-\tfrac12, n-\tfrac12\,; 1\,; 1) } \int_{-r}^{r} |g^{(m)}(x)|^2 \, \d x.
\end{align}
This inequality is sharp. The value of $(\mathbb{EP}1)(m-\tfrac12, n-\tfrac12\,; 1\,; 1)$ is given by Theorem \ref{Teoremacco_versao_dB_Intro}.
\end{corollary}	
\noindent {\sc Remark:} When $r = \tfrac{1}{2\pi}$, the extremizers of \eqref{20230223_18:13} are the Fourier transforms of the extremizers of Theorem \ref{Thmaco_class} , applied to the situation of Theorem \ref{Teoremacco_versao_dB_Intro} (with $\alpha = m - \tfrac12$ and $\beta = n - \tfrac12$); the extremizers for general $r$ are obtained via a suitable dilation.

\smallskip 

Poincaré inequalities have been extensively studied in many different contexts within PDEs and calculus of variations. In fact, Corollary \ref{20230224_13:44} has been previously obtained via different methods. The cases $(m,n) = (1,0), (2,1)$ date back to Steklov \cite{Stek}. The cases $(n+1,n)$, with $n\geq 2$, date back to Janet \cite{Jan}; see also \cite{NP}. The sharp constant in the generic integer case $(m,n)$ has been recently characterized in the work of Yu. P. Petrova \cite{Petrova} (see also \cite[Theorem 7]{NSc}), also as the smallest positive solution of a certain explicit determinant equation. Petrova's determinant is different than ours (in particular, it has larger order $k = m-n$ and relies on some specific properties of Bessel functions for certain simplifications) but, of course, it must lead to the same answer in these particular cases. The formulation in Theorem \ref{Teoremacco_versao_dB_Intro} has the advantage of being amenable to a generalization to the broader setup of de Branges spaces in Theorem \ref{Teoremacco_versao_dB}. For a detailed account on the history of the inequalities \eqref{20230223_18:13} (also known as Steklov-type inequalities) and some related variants, we refer the reader to the survey article by A. I. Nazarov and A. P. Shcheglova \cite{NSc} and the references therein.

\smallskip

It is also possible to recast our multidimensional extremal problem (EP1), via the Fourier transform, in a problem involving the classical gradient and Laplacian.
\begin{corollary}\label{cor_Laplacian}
Let $d \geq 2$. Let $m, n \in \Z_{\geq0}$, $m_1,n_1 \in \{0,1\}$ be such that $2m +m_1 \geq 2n + n_1$. For any $g \in W^{2m+m_1,2}_0(B_r)$ we have
\begin{align}\label{20230224_15:20}
\int_{B_r} \big|\nabla^{n_1}\big(\Delta^{n}g\big) (\x)\big|^2 \,\d \x \leq \frac{r^{2(2m + m_1 -2n - n_1)}}{(\mathbb{EP}1)(2m+m_1-1+\tfrac{d}{2}, 2n+n_1-1+\tfrac{d}{2}\,; 1\,; 1) } \int_{B_r} \big|\nabla^{m_1}\big(\Delta^{m}g\big)(\x)\big|^2 \, \d \x.
\end{align}
This inequality is sharp. The value of $(\mathbb{EP}1)(2m\!+\!m_1\!-\!1+\tfrac{d}{2}, 2n\!+\!n_1\!-\!1+\tfrac{d}{2}\,; 1\,; 1)$ is given in Theorem \ref{Teoremacco_versao_dB_Intro}.
\end{corollary}
\noindent {\sc Remark:} In \eqref{20230224_15:20} note that we have used Theorem \ref{Thm_dimension_shifts} and \eqref{20220816_15:18_1}. When $r = \tfrac{1}{2\pi}$, the radial extremizers of \eqref{20230224_15:20} are given by the Fourier transforms of the lifts (see \S \ref{Lifts_Sec} and \S \ref{Exist_Extre_Subsc}) of the extremizers of Theorem \ref{Thmaco_class}, applied to the situation of Theorem \ref{Teoremacco_versao_dB_Intro} (with $\alpha = 2m+m_1-1+\tfrac{d}{2}$ and $\beta =2n+n_1-1+\tfrac{d}{2}$); the radial extremizers for general $r$ are obtained via a suitable dilation.

\smallskip

In the particular situations of Corollaries \ref{20230224_13:44} and \ref{cor_Laplacian}, it turns out that Poincaré inequalities and Fourier uncertainty are simply the two sides of the same coin. Our takeaway point here is that the de Branges space approach on the entire function side provides an alternative framework to look into the classical problem of finding the sharp constants in these Poincaré inequalities and, in fact, allows for a robust generalization of it.

\subsection{Application 2: a related extremal problem} The classical Beurling-Selberg extremal problem in approximation theory can be described as follows: given a function $f:\R \to \R$, one wants to find a real entire function $M$ of prescribed exponential type such that $M(x) \geq f(x)$ for all $x \in \R$, with the norm $\|M - f\|_{L^1(\R)}$ minimal. The original situation considered by Beurling in the late 1930s was with $f(x) = \sgn(x)$. This type of problem has been extensively studied in the literature with several interesting applications in analytic number theory, e.g. \cite{CCM, CarChi, CL2, CLV, ChSo, HV, V}. Recently, the first author and Littmann \cite{CL3} studied a variant of the Beurling-Selberg extremal problem for $f(x) = \sgn(x)$, with the additional constraint that the majorant $M$ is non-decreasing in $(-\infty, 0)$ and non-increasing in $(0,\infty)$, and used such an extremal function to give a simple Fourier analysis proof of the (non-sharp) weighted Hilbert's inequality, introduced by Montgomery and Vaughan \cite{MV}.

\smallskip

The following related question arose in discussions between the first author and J. D. Vaaler in 2015. This is the analogue, with an additional monotonicity constraint, of a question solved by Holt and Vaaler in \cite{HV}.

\smallskip

\noindent {\it Extremal Problem 2} (EP2) $($Radial non-increasing delta majorant$)$: Let $\mc{R^+}(d\, ; 2\delta)$ be the class of real entire functions $M:\C^d \to \C$ that verify the following properties: (i) $M$ has exponential type at most $2\delta$; (ii) $M$ is non-negative and radial non-increasing on $\R^d$; (iii) $M({\bf 0}) \geq 1$. Find the value of 
\begin{equation*}
(\mathbb{EP}2)(d\, ; \delta) := \inf_{M \in \mc{R}^+(d;2\delta)}\int_{\R^d} M(\x)\,\d\x.
\end{equation*}

\smallskip We show that the problem (EP2) is related to the problem (EP1) in the following way. 

\begin{theorem}\label{RBS=FU}
There exist extremizers for $(\mathbb{EP}2)(d\, ; \delta)$ and 
$$(\mathbb{EP}2)(d\, ; \delta)  = \frac{\omega_{d-1}}{d} \,  (\mathbb{EP}1)(\tfrac{d}{2}, 0\, ; 1\,; \delta)\,,$$
where $\omega_{d-1} = 2 \, \pi^{d/2} \, \Gamma(d/2)^{-1}$ is the surface area of the unit sphere $\mathbb{S}^{d-1}\subset \R^d$.
\end{theorem}
From Theorem \ref{Teoremacco_versao_dB_Intro} we have the exact answer of the extremal problem (EP2) when the dimension $d$ is even. The characterization of the extremizers is obtained via Theorem \ref{Thmaco_class} and the radial symmetrization considerations of Theorem \ref{Thm_dimension_shifts}. Here is a table with the first values when $\delta =1$:

\medskip

\begin{center}
        \begin{tabular}{ |c | c | c |c|c|c|c|c|c|}
        \hline
             & $d=2$ & $d=4$ &$d=6$ & $d=8$ & $d=10$ & $d=12$&$d=14$ &$d=16$\\
            \hline
            $\big((\mathbb{EP}2)(d\, ; 1) \big)^{1/d}$ & 4.26\ldots & $4.76\ldots$ &$5.23\ldots$ & $5.66\ldots$ &$6.07\ldots$ & $6.45\ldots$ & $6.81\ldots$&$7.15\ldots$\\
      
        \hline
        \end{tabular}
 \end{center} 
 
\medskip

\noindent Finding the exact answer in the cases when $d$ is odd seems to be a subtler task. For a conjecture related to the case $d=1$, see the recent work of Chirre, Dimitrov, Quesada-Herrera and Sousa \cite{ChiDimQH}. 

\medskip

From the asymptotics in Theorem \ref{Thm5 - Asympt} and Stirling's formula for the Gamma function one obtains
\begin{align*}
\log \big((\mathbb{EP}2)(d\, ; 1) \big) = \frac{d}{2}\log d + O(d)
\end{align*}
as $d \to \infty$. To put in perspective, if we let ${\rm (EP2)}^*$ be the classical one-delta problem (i.e. the same problem as (EP2), but without the radial non-increasing constraint), and let $(\mathbb{EP}2)^*(d\, ; \delta)$ be the corresponding infimum, from the work of Holt and Vaaler \cite{HV} one has the exact answer $(\mathbb{EP}2)^*(d\, ; 1) = d \, 2^{d-1}\, \pi^{d/2}\, \Gamma(d/2)$. In particular, from Stirling's formula, 
\begin{align*}
\log \big((\mathbb{EP}2)^*(d\, ; 1) \big) = \frac{d}{2}\log d + O(d)\,,
\end{align*}
and one concludes that the additional monotonicity constraint does not change the main order term in this logarithmic scale.

\subsection{Notation} A function $F:\C^d \to \C$ is said to be {\it real entire} if it is entire and its restriction to $\R^d$ is real-valued. If $F :\C^d \to \C$ is entire, we define the entire function $F^*:\C^d \to \C$ by $F^*(\z) := \overline{F(\overline{\z})}$. We denote by $\C^+ = \{z \in \C \ ; \ {\rm Im}(z) >0\}$ the open upper half-plane. Throughout the text we work with the constants $\omega_{d-1} := 2 \, \pi^{d/2} \, \Gamma(d/2)^{-1}$ (surface area of the unit sphere $\mathbb{S}^{d-1}\subset \R^d$) and $c_\nu := \pi \,2^{-2\nu-1}\, \Gamma(\nu+1)^{-2}$.

\section{Preliminaries}
In this section we collect a few auxiliary results that shall be relevant for our purposes later on. In doing so, we also establish some of the notation and terminology that will be used. Some of the following results are direct quotes from the current literature that we state as lemmas for the convenience of the reader.

\subsection{Fourier uncertainty} We first highlight the following particular version of Fourier uncertainty that follows from classical results.
\begin{lemma} \label{Lem9 - FU} If $f \in L^2(\R)$ and ${\rm supp}(\widehat{f}) \subset [-\tfrac{1}{2\pi}, \tfrac{1}{2\pi}]$ we have 
\begin{align}\label{20230201_19:40}
 \int_{[-s,s]} |f(x)|^2\,\d x \leq e^{Cr} \left(1 - e^{-Cs}\right) \int_{[-r,r]^c} |f(x)|^2\,\d x
\end{align}
for any $r,s >0$, where $C$ is a universal constant. 
\end{lemma}
\begin{proof}
We first claim that it suffices to prove that, for any $r>0$, 
\begin{equation}\label{20230201_19:41}
\int_{\R} |f(x)|^2\,\d x \leq  e^{Cr} \int_{[-r,r]^c} |f(x)|^2\,\d x,
\end{equation}
which in turn is equivalent to 
\begin{equation*}
\int_{[-r,r]} |f(x)|^2\,\d x \leq  (e^{Cr}-1) \int_{[-r,r]^c} |f(x)|^2\,\d x.
\end{equation*}
Hence, for numbers $r,s \geq 0$ we have
\begin{align*}
e^{Cr} \int_{[-r,r]^c} |f(x)|^2\,\d x & \geq \int_{\R} |f(x)|^2\,\d x = \int_{[-s,s]} |f(x)|^2\,\d x + \int_{[-s,s]^c} |f(x)|^2\,\d x \\
& \geq  \left( 1 + \frac{1}{e^{Cs} - 1}\right) \int_{[-s,s]} |f(x)|^2\,\d x\,,
\end{align*}
which plainly implies \eqref{20230201_19:40}.

\smallskip

From now on we slightly abuse the notation and the constant $C$ can possibly change from line to line. In order to arrive at \eqref{20230201_19:41}, recall first the Fourier uncertainty principle of Amrein and Berthier \cite{AB} and Nazarov \cite{N} (see also the work of Jaming \cite{J}): for $E, F \subset \R$ sets of finite Lebesgue measure and $f \in L^2(\R)$ we have
\begin{equation*}
\int_{\R} |f(x)|^2\,\d x \leq C \, e^{C  |E|\, |F|} \left( \int_{E^c} |f(x)|^2\,\d x  + \int_{F^c} |\widehat{f}( x)|^2\,\d x\right).
\end{equation*}
Specializing to our case, with $E= [-r,r]$ and $F = [-\tfrac{1}{2\pi}, \tfrac{1}{2\pi}]$, we get 
\begin{equation}\label{20230201_20:17}
\int_{\R} |f(x)|^2\,\d x \leq C \, e^{Cr}  \int_{[-r,r]^c} |f(x)|^2\,\d x.
\end{equation}
For a better bound when $r$ is small, we recall the Fourier uncertainty principle of Donoho and Stark \cite[Theorem 2]{DS}: for $E, F \subset \R$ measurable sets and $f \in L^2(\R)$ we have
$$\|f\|_2 - \|f . \chi_{_{E^c}}\|_2 -  \|\widehat{f} . \chi_{_{F^c}}\|_2 \leq|E|^{1/2}\, |F|^{1/2} \, \|f\|_2.$$
Choosing $E= [-r,r]$ and $F = [-\tfrac{1}{2\pi}, \tfrac{1}{2\pi}]$ we get, for $r$ small,
\begin{equation}\label{20230201_20:18}
\int_{\R} |f(x)|^2\,\d x \leq \frac{1}{(1 - C\sqrt{r})^2} \int_{[-r,r]^c} |f(x)|^2\,\d x.
\end{equation}
Note that $(1 - C\sqrt{r}))^{-2} = 1 + O(\sqrt{r})$ for $r$ small. Hence, \eqref{20230201_20:17} and \eqref{20230201_20:18} lead to \eqref{20230201_19:41}.
\end{proof}

\subsection{De Branges spaces} \label{Sub_DB_spaces}Our extremal problem (EP1) is related to the  the beautiful theory of de Branges spaces of entire functions \cite{Branges}, that we now briefly review. 

\subsubsection{Overview} \label{GEn_Ov_dB} Given a {\it Hermite-Biehler} function $E: \C \to \C$, i.e. an entire function that verifies
\begin{equation}\label{20230315_15:01}
|E^*(z)| < |E(z)|
\end{equation} 
for all $z \in \C^+$, the de Branges space $\mc{H}(E)$ associated to $E$ is the space of entire functions $f:\C \to \C$ such that
\begin{equation*}
\|f\|_{\mc{H}(E)}^2 := \int_{\R} |f(x)|^{2} \, |E(x)|^{-2} \, \d x <\infty\,,
\end{equation*}
and such that $f/E$ and $f^*/E$ have bounded type and non-positive mean type\footnote{A function $f$, analytic in $\C^{+}$, has {\it bounded type} if it can be written as the quotient of two functions that are analytic and bounded in $\C^{+}$. If $f$ has bounded type in $\C^{+}$, from its Nevanlinna factorization \cite[Theorems 9 and 10]{Branges} one has $v(f) := \limsup_{y \to \infty} \, y^{-1}\log|f(iy)| <\infty$. The number $v(f)$ is called the {\it mean type} of $f$.} in $\C^{+}$. This turns out to be a reproducing kernel Hilbert space with inner product given by
\begin{equation*}
\langle f, g \rangle_{\mc{H}(E)} :=  \int_{\R}f(x) \, \ov{g(x)} \, |E(x)|^{-2} \, \d x.
\end{equation*}
Associated to $E$, one can consider a pair of real entire functions $A$ and $B$ such that $E(z) = A(z) -iB(z)$. These companion functions are given by
\begin{equation*}
A(z) := \frac12 \big(E(z) + E^*(z)\big) \ \ \ {\rm and}\ \ \  B(z) := \frac{i}{2}\big(E(z) - E^*(z)\big)\,, 
\end{equation*}
and note that they can only have real roots, by the Hermite-Biehler condition. The reproducing kernel, that we denote by $K(w,\cdot)$, is given by (see \cite[Theorem 19]{Branges})
\begin{equation}\label{20210913_13:57}
K(w,z) = \frac{B(z)A(\ov{w}) - A(z)B(\ov{w})}{\pi (z - \ov{w})}\,,
\end{equation}
and, when $z = \ov{w}$, one has
\begin{equation}\label{Intro_Def_K}
K(\ov{z}, z) = \frac{B'(z)A(z) - A'(z)B(z)}{\pi }.
\end{equation}

The set of functions $\Gamma_A:= \{K(\xi, \cdot)\ ;  A(\xi) = 0\}$ is always an orthogonal set in $\mc{H}(E)$. If $A \notin \mc{H}(E)$, the set $\Gamma_A$ is an orthogonal basis of $\mc{H}(E)$ and, if $A \in \mc{H}(E)$, the only elements of $\mc{H}(E)$ that are orthogonal to $\Gamma_A$ are the constant multiples of $A$. The same statements hold for the set $\Gamma_B:= \{K(\xi, \cdot)\ ;  B(\xi) = 0\}$ and these are specializations of a more general theorem in the theory; see \cite[Theorem 22]{Branges}. In particular, if $A \notin \mc{H}(E)$, for every $f \in \mc{H}(E)$ we have
\begin{equation}\label{20210809_11:01}
f(z) = \sum_{A(\xi) = 0}  \frac{f(\xi)}{K(\xi, \xi)} \, K(\xi, z) \ \ \ {\rm and} \ \ \   \|f\|_{\mc{H}(E)}^2 = \sum_{A(\xi) = 0}  \frac{\big| f(\xi)\big|^2 }{K(\xi, \xi)}.
\end{equation}
Analogous formulas would hold with $A$ replaced by $B$. These formulas are fundamental for our purposes, as they provide a general analogue of the theory of Fourier series to this broad setup.

\subsubsection{A class of homogeneous de Branges spaces} \label{hom_dB_subsec} Let $\nu > -1$. A de Branges space $\mc{H}(E)$ is said to be homogeneous of order $\nu$ if, for all $0 < a < 1$ and all $f \in \mc{H}(E)$, the function $z \mapsto a^{\nu +1}f(az)$ belongs to $\mc{H}(E)$ and has the same norm as $f$.  Such spaces were characterized by L. de Branges in \cite{B2} (see also \cite[Section 50]{Branges}). 

\smallskip

For $\nu >-1$, let $A_{\nu}$ and $B_{\nu}$ be the real entire functions defined in \eqref{20220816_17:43} and \eqref{20221121_10:21}. The function 
\begin{equation*}
E_{\nu}(z) = A_{\nu}(z) - iB_{\nu}(z)
\end{equation*}
turns out to be a Hermite-Biehler function with no real zeros and $\mc{H}(E_{\nu})$ is homogeneous of order $\nu$. Moreover we have $\tau(A_{\nu}) = \tau(B_{\nu}) = \tau(E_{\nu}) =1$. Note also that $A_{\nu}, B_{\nu} \notin \mc{H}(E_{\nu})$ (just observe the behaviour of the Bessel function $J_{\nu}$ at infinity) and hence the formulas in \eqref{20210809_11:01} hold. See \cite[Section 50]{Branges}, \cite[Section 5]{HV} or \cite[Sections 3 and 4]{CL1} for further details. We gather other relevant facts about the spaces $\mc{H}(E_{\nu})$ in the next lemma, which is contained in \cite[Eqs. (5.1), (5.2) and Lemma 16]{HV}. 

\begin{lemma}\label{Sec5_rel_facts}
Let $\nu > -1$. The following properties hold:
\begin{enumerate}
\item[(i)] There exist positive constants $a_\nu,b_\nu$ such that 
\begin{align*}
a_\nu |x|^{2\nu+1} \le |E_{\nu}(x)|^{-2} \le b_\nu |x|^{2\nu+1}
\end{align*}
for all $x \in \R$ with $|x|\geq1$.
\smallskip

\item[(ii)] For $f\in\H(E_\nu)$ we have the identity 
\begin{align} \label{Lem17_ii}
\int_{\R} |f(x)|^{2}\,|E_{\nu}(x)|^{-2}\, \d x = c_\nu \int_{\R} |f(x)|^2 \,|x|^{2\nu+1} \,\d x\,,
\end{align}
with $c_\nu = \pi \,2^{-2\nu-1}\, \Gamma(\nu+1)^{-2}$.

\smallskip

\item[(iii)] An entire function $f$ belongs to $\H(E_\nu)$ if and only if $f$ has exponential type at most $1$ and
\begin{equation*}
\int_{\R} |f(x)|^2 \,|x|^{2\nu+1}\, \d x <\infty.
\end{equation*}
\end{enumerate}
\end{lemma}
By Lemma \ref{Sec5_rel_facts} note that our space $\mc{H}_{\nu}(1\,; 1)$ defined in \eqref{20221121_11:09} is exactly the same as $\mc{H}(E_{\nu})$, with norms differing just by a multiplicative constant $c_{\nu}$ given in \eqref{Lem17_ii}. This is the key identity that connects the theory of de Branges spaces to our weighted Paley-Wiener spaces.

\subsection{Lifts and radial symmetrization} \label{Lifts_Sec} We now recall two basic constructions from the work of Holt and Vaaler \cite[Section 6]{HV}.

\smallskip

If $f:\C\to \C$ is an even entire function with power series representation
\begin{align*}
f(z) = \sum_{k=0}^\infty c_k z^{2k}\,,
\end{align*}
we define the lift ${\mc L}_d(f):\C^d \to \C$ by
\begin{align*}
{\mc L}_d(f)({\z}) := \sum_{k=0}^\infty c_k (z_1^2 + \ldots + z_d^2)^k.
\end{align*}
For an entire function $F: \C^d \to \C$, with $d >1$, we define its radial symmetrization $\widetilde{F}:\C^d \to \C$ by 
\begin{equation*}
\widetilde{F}(\z) := \int_{SO(d)}F(R\z)\,\d \sigma(R),
\end{equation*}
where $SO(d)$ denotes the compact topological group of real orthogonal $d \times d$ matrices with determinant $1$, with associated Haar measure $\sigma$ normalized so that $\sigma(SO(d)) = 1$. If $d=1$ we simply set $\widetilde{F}(z) := \tfrac12\{F(z) + F(-z)\}$. The next result is a compilation of \cite[Lemmas 18 and 19]{HV}.

\begin{lemma}\label{type-to-ntype} 
The following propositions hold:
\smallskip
\begin{itemize}
\item[(i)] Let $f: \C \to \C$ be an even entire function. Then $f$ has exponential type if and only if ${\mc L}_d(f)$ has exponential type, and $\tau(f) = \tau({\mc L}_d(f))$. 

\smallskip

\item[(ii)] Let $F: \C^d \to \C$ be an entire function. Then $\widetilde{F}:\C^d \to \C$ is an entire function with power series expansion of the form
\begin{equation*}
\widetilde{F}({\z}) = \sum_{k=0}^\infty c_k (z_1^2 + \ldots + z_d^2)^k.
\end{equation*}
Moreover, if $F$ has exponential type then $\widetilde{F}$ has exponential type and $\tau\big(\widetilde{F}\big) \leq \tau(F)$.

\end{itemize}
\end{lemma}

\subsection{Krein's decomposition} We now recall a suitable version of a classical result of Krein on the decomposition of entire functions of exponential type that are non-negative on $\R$; see \cite[p. 154]{A}. The version presented below is a particular case of a more general version for de Branges spaces; see \cite[Lemma 14]{CL1}.
\begin{lemma}[Krein's decomposition]\label{Krein_dec_lem}
Let $\nu > -1$ and $\delta >0$. Let $f$ be a real entire function of exponential type at most $2\delta$, that is non-negative on $\R$ and belongs to $L^1 (\R, |x|^{2\nu +1}\,\d x)$. Then there exists $g\in \mc{H}_{\nu}(1\,; \delta)$ such that
\begin{align*}
f(z) = g(z) \, g^*(z)
\end{align*}
for all $z \in \C$.
\end{lemma}

\subsection{Estimates for Bessel functions} The following result is classical in the theory of Bessel functions. We include a short proof for convenience. 

\begin{lemma}
Let $\nu >-1$. For any $0 \leq x \leq 2 \sqrt{(\nu + 1)}$ we have
\begin{equation}\label{20230207_16:55}
0 \leq 1 - \frac{x^2}{4(\nu +1)} \leq A_{\nu}(x) \leq  1 - \frac{x^2}{4(\nu +1)} + \frac{x^4}{32(\nu +1)(\nu +2)} \leq 1. 
\end{equation}
Also, for any $0 \leq x \leq 2 \sqrt{(\nu + 2)}$ we have
\begin{equation}\label{20230207_17:14}
0 \leq \frac{x}{2(\nu+1)} - \frac{x^3}{8(\nu +1)(\nu+2)} \leq B_{\nu}(x) \leq  \frac{x}{2(\nu+1)}.
\end{equation}
\end{lemma} 
\begin{proof}
The proof is by comparing consecutive terms in the series expansions \eqref{20220816_17:43} and \eqref{20221121_10:21}. For instance, in order to show  \eqref{20230207_16:55} we compare the terms $n= 2k$ and $n = 2k+1$ in the series \eqref{20220816_17:43}. We observe that 
\begin{align*}
\frac{(-1)^{2k} \big(\tfrac12 x\big)^{4k}}{(2k)!(\nu +1)(\nu +2)\ldots(\nu+2k)} \geq - \frac{(-1)^{2k+1} \big(\tfrac12 x\big)^{4k+2}}{(2k+1)!(\nu +1)(\nu +2)\ldots(\nu+2k+1)} 
\end{align*}
holds for $x \geq 0$ if and only if
\begin{align*}
1 \geq \frac{x^2}{4(2k+1)(\nu + 2k +1)}.
\end{align*}
In the worst case scenario, i.e. when $k=0$, this holds if $x \leq 2\sqrt{(\nu + 1)}$. This plainly leads us to the inequalities on the left-hand side of \eqref{20230207_16:55}. The inequality on the right-hand side of \eqref{20230207_16:55} is proved analogously, comparing the terms $n= 2k+1$ and $n = 2k+2$, for $k \geq 1$, in the series \eqref{20220816_17:43}. Similarly, one arrives at \eqref{20230207_17:14} via the series \eqref{20221121_10:21}.
\end{proof}

The next result is a basic upper bound that suffices for our purposes in \S \ref{Sec_LBII}. For $\nu >-1$, let $K_{\nu}$ be the reproducing kernel of the Hilbert space $\mc{H}(E_{\nu})$ defined in \S \ref{Sub_DB_spaces}.

\begin{lemma}\label{Lem14_UBRK}
For $-1 < \nu\leq 0$ we have, for $0 \leq x \leq 2\sqrt{\nu +1}$,
\begin{align}\label{20230208_16:13}
\pi K_{\nu}(x,x) \leq \frac{1}{2 (\nu +1)} + \frac{x^2}{4(\nu +1)^2}.
\end{align}
\end{lemma}
\begin{proof}
From \eqref{Intro_Def_K} and \eqref{20221130_14:53} we have
\begin{align*}
\pi K_{\nu}(x,x) = A_{\nu}(x)^2 +  B_{\nu}(x)^2 - \frac{(2\nu +1)A_{\nu}(x)B_{\nu}(x) }{x}.
\end{align*}
When $-1 < \nu \leq -\frac12$, then $2\nu +1 \leq 0$ and may use the upper bounds in \eqref{20230207_16:55} and \eqref{20230207_17:14} to get 
\begin{align*}
\pi K_{\nu}(x,x)  \leq 1 + \frac{x^2}{4(\nu+1)^2}    - \frac{(2\nu+1)}{2(\nu +1)} =  \frac{1}{2 (\nu +1)} + \frac{x^2}{4(\nu +1)^2}.
\end{align*}
If $-\frac12 < \nu \leq 0$, then $2\nu +1 > 0$ and we may use again \eqref{20230207_16:55} and \eqref{20230207_17:14} (now both upper and lower bounds) to get

\begin{align*}
\pi K_{\nu}(x,x) & \leq \left(1 - \frac{x^2}{4(\nu +1)} + \frac{x^4}{32(\nu +1)(\nu +2)}\right) +  \frac{x^2}{4(\nu+1)^2} \\
&  \ \ \ \ \ \ \ \ \ \ \ \ \ \ \ - (2\nu+1) \left(1 - \frac{x^2}{4(\nu +1)}\right) \left(\frac{1}{2(\nu+1)} - \frac{x^2}{8(\nu +1)(\nu+2)}\right) \\
& = \frac{1}{2 (\nu +1)} + \frac{x^2}{4(\nu +1)^2} + \frac{x^2}{32(\nu+1)^2(\nu+2)}\left(8\nu^2 + 8\nu - 4 - \nu x^2\right).
\end{align*}
Observe that in this range we have $8\nu^2 + 8\nu - 4 - \nu x^2\leq 0$, which leads us to \eqref{20230208_16:13}.
\end{proof}

\section{Dimension shifts: proof of Theorem \ref{Thm_dimension_shifts}} \label{Sec_proof_Thm1_Dim_shifts}
We now introduce an extremal problem that is closely related to our problem (EP1).

\smallskip

\noindent {\it Extremal Problem 3} (EP3): Let $\alpha \geq \beta > -1$ and $\delta >0$ be real parameters, and $d \in \N$. Let  $\mc{W}^+_{\alpha}(d\,; 2\delta)$ be the set of real entire functions $M: \C^d \to \C$ of exponential type at most $2\delta$ that are non-negative on $\R^d$ and such that 
$$\int_{\R^d} M(\x)\, |\x|^{2\alpha + 2 -d}\,\d\x <\infty.$$
Find
\begin{equation}\label{20221116_11:19}
(\mathbb{EP}3)(\alpha, \beta\,; d\,; \delta):= \inf_{ 0 \neq M \in  \mc{W}^+_{\alpha}(d; 2\delta) } \frac{\int_{\R^d} M(\x) \, |\x|^{2\alpha + 2 -d}\,\d\x}{\int_{\R^d} M(\x) \, |\x|^{2\beta + 2 -d}\,\d\x}.
\end{equation}

\smallskip

We now proceed to the proof of Theorem \ref{Thm_dimension_shifts}.

\subsubsection{Step 1} Observe first that if $F \in \mc{H}_{\alpha}(d; \delta)$ then $M(\z) = F(\z)F^*(\z) \in \mc{W}^+_{\alpha}(d; 2\delta)$. Hence
\begin{equation}\label{20221116_20:08}
(\mathbb{EP}1)(\alpha, \beta\,; d\,; \delta) \geq (\mathbb{EP}3)(\alpha, \beta\,; d\,; \delta).
\end{equation}
\subsubsection{Step 2} If $M \in \mc{W}^+_{\alpha}(d; 2\delta)$, by Lemma \ref{type-to-ntype} (ii) we have that $\widetilde{M} \in \mc{W}^+_{\alpha}(d; 2\delta)$ and an application of Fubini's theorem yields
\begin{equation*}
\int_{\R^d} M(\x) \, |\x|^{2\gamma + 2 -d}\,\d\x = \int_{\R^d} \widetilde{M}(\x) \, |\x|^{2\gamma + 2 -d}\,\d\x
\end{equation*}
for any $\gamma > -1$. In particular, we can restrict our search for the infimum in \eqref{20221116_11:19} to functions $M \in \mc{W}^+_{\alpha}(d; 2\delta)$ that are radial on $\R^d$. By Lemma \ref{type-to-ntype}, such $M$ is the lift ${\mc L}_d(f)$ of an even entire function $f \in \mc{W}^+_{\alpha}(1; 2\delta)$, and conversely. Reducing to polar coordinates we have
$$\frac{\int_{\R^d} M(\x) \, |\x|^{2\alpha + 2 -d}\,\d\x}{\int_{\R^d} M(\x) \, |\x|^{2\beta + 2 -d}\,\d\x} = \frac{\int_{\R} f(x) \, |x|^{2\alpha + 1}\,\d x}{\int_{\R} f(x) \, |x|^{2\beta + 1}\,\d x}\,,$$
and we conclude that 
\begin{equation}\label{20221116_20:091}
(\mathbb{EP}3)(\alpha, \beta\,; d\,; \delta) = (\mathbb{EP}3)(\alpha, \beta\,; 1\,; \delta).
\end{equation}
\subsubsection{Step 3} By Krein's decomposition (Lemma \ref{Krein_dec_lem} above), every $f \in \mc{W}^+_{\alpha}(1; 2\delta)$ can be written as $f(z) = g(z)\, g^*(z)$ with $g \in \mc{H}_{\alpha}(1; \delta)$, and conversely. This implies that 
\begin{equation}\label{20221116_20:092}
(\mathbb{EP}3)(\alpha, \beta\,; 1\,; \delta) = (\mathbb{EP}1)(\alpha, \beta\,; 1\,; \delta).
\end{equation}

\subsubsection{Step 4} \label{Sec_Step4_Pf1} We now use a fact that will be proved independently in the next section: there exists an even extremizer $g \in \mc{H}_{\alpha}(1; \delta)$ for $(\mathbb{EP}1)(\alpha, \beta\,; 1\,; \delta)$. By Lemma \ref{type-to-ntype} (i), we may lift it to ${\mc L}_d(g) \in  \mc{H}_{\alpha}(d; \delta)$, and we find that 
\begin{equation}\label{20221116_19:55}
(\mathbb{EP}1)(\alpha, \beta\,; 1\,; \delta)= \frac{\int_{\R} |g(x)|^2 \, |x|^{2\alpha + 1}\,\d x}{\int_{\R} |g(x)|^2 \, |x|^{2\beta + 1}\,\d x} = \frac{\int_{\R^d} |{\mc L}_d(g)(\x)|^2 \, |\x|^{2\alpha + 2-d}\,\d \x}{\int_{\R^d} |{\mc L}_d(g)(\x)|^2 \, |\x|^{2\beta + 2-d}\,\d \x} \geq (\mathbb{EP}1)(\alpha, \beta\,; d\,; \delta).
\end{equation}

\subsubsection{Conclusion} Combining \eqref{20221116_20:08}, \eqref{20221116_20:091}, \eqref{20221116_20:092} and \eqref{20221116_19:55} we arrive at 
$$(\mathbb{EP}1)(\alpha, \beta\,; d\,; \delta) \geq (\mathbb{EP}3)(\alpha, \beta\,; d\,; \delta) =  (\mathbb{EP}3)(\alpha, \beta\,; 1\,; \delta) = (\mathbb{EP}1)(\alpha, \beta\,; 1\,; \delta) \geq (\mathbb{EP}1)(\alpha, \beta\,; d\,; \delta),$$
and hence we must have the equalities
$$(\mathbb{EP}1)(\alpha, \beta\,; d\,; \delta) = (\mathbb{EP}3)(\alpha, \beta\,; d\,; \delta) =  (\mathbb{EP}3)(\alpha, \beta\,; 1\,; \delta) = (\mathbb{EP}1)(\alpha, \beta\,; 1\,; \delta).$$

\section{Existence of extremizers: proof of Theorem \ref{Thm2}}\label{Sec_Existence_of _extremizers}

\subsection{Existence of extremizers in dimension $d=1$.} \label{Exist_extr_dim_1}We now want to establish the existence of extremizers for $(\mathbb{EP}1)(\alpha, \beta\,; 1\,; \delta)$. By dilation, it suffices to consider $\delta =1$. To shorten the notation, let us write $(\mathbb{EP}1)(\alpha, \beta\,; 1\,; 1)$ in this subsection simply as $(\mathbb{EP}1)$. Assume also that $\alpha >\beta$.

\smallskip

Let $\{f_n\}_{n\geq 1} \subset  \mc{H}_{\alpha}(1\,; 1)$ be an extremizing sequence, normalized so that $\|f_n\|_{ \mc{H}_{\alpha}(1; 1)} = 1$. This means that 
$$\|f_n\|^2_{ \mc{H}_{\beta}(1\,; 1)} \to (\mathbb{EP}1)^{-1}.$$
By reflexivity, up to a subsequence, we may assume that $f_n \rightharpoonup g$ for some $g \in \mc{H}_{\alpha}(1\,; 1)$. Note that 
\begin{equation}\label{20221117_10:45}
\|g\|_{ \mc{H}_{\alpha}(1; 1)} \leq 1.
\end{equation}
Since $\mc{H}_{\alpha}(1\,; 1) = \mc{H}(E_{\alpha})$ as sets, with norms differing by a multiplicative constant (Lemma \ref{Sec5_rel_facts}), this is a reproducing kernel Hilbert space; see \S \ref{Sub_DB_spaces}. Then, for any $w \in \C$, we have the pointwise convergence
\begin{align}\label{20230221_15:24}
f_n(w) = \langle f_n, K_{\alpha}(w, \cdot) \rangle_{\mc{H}(E_{\alpha})} \to  \langle g, K_{\alpha}(w, \cdot) \rangle_{\mc{H}(E_{\alpha})} = g(w).
\end{align}
By the Cauchy-Schwarz inequality, we also get that 
\begin{align}\label{20230221_15:39}
|f_n(w)| = |\langle f_n, K_{\alpha}(w, \cdot) \rangle_{\mc{H}(E_{\alpha})}| \leq \|f_n\|_{ \mc{H}(E_{\alpha})} \,  \|K_{\alpha}(w, \cdot)\|_{\mc{H}(E_{\alpha})} =  c_{\alpha}^{1/2} \, K_{\alpha}(w,w)^{1/2},
\end{align}
and since $K_{\alpha}$ is a continuous function of two variables, we get that $f_n$ is uniformly bounded in compact subsets of $\C$. 

\smallskip

Given any $a < (\mathbb{EP}1)^{-1}$, there exists an $N_0 = N_0(a)$ such that for $n\geq N_0$ we have
\begin{align}
\begin{split}\label{20221117_10:33}
a & \leq \int_{\R} |f_n(x)|^2\,|x|^{2 \beta +1}\,\d x  \\
& = \int_{|x| \leq R} |f_n(x)|^2\,|x|^{2 \beta +1}\,\d x + \int_{|x| > R} |f_n(x)|^2\,|x|^{2 \beta +1}\,\d x\\
& \leq \int_{|x| \leq R} |f_n(x)|^2\,|x|^{2 \beta +1}\,\d x + R^{2\beta - 2\alpha}\int_{|x| > R} |f_n(x)|^2\,|x|^{2 \alpha +1}\,\d x\\
& \leq \int_{|x| \leq R} |f_n(x)|^2\,|x|^{2 \beta +1}\,\d x  + R^{2\beta - 2\alpha}\,,
\end{split}
\end{align}
for any fixed $R$, where we have used that $\alpha > \beta$ and $\|f_n\|_{ \mc{H}_{\alpha}(1; 1)} = 1$. Letting $n \to \infty$ and applying the dominated convergence theorem in \eqref{20221117_10:33} we get
\begin{equation}\label{20221117_10:34}
a - R^{2\beta - 2\alpha} \leq \int_{|x| \leq R} |g(x)|^2\,|x|^{2 \beta +1}\,\d x \leq \int_{\R} |g(x)|^2\,|x|^{2 \beta +1}\,\d x.
\end{equation}
Letting $R \to \infty$ in \eqref{20221117_10:34} we arrive at 
\begin{equation*}
a  \leq \int_{\R} |g(x)|^2\,|x|^{2 \beta +1}\,\d x.
\end{equation*}
This shows that $0 \neq g $ and, since  $a < (\mathbb{EP}1)^{-1}$ is arbitrary, we conclude that 
\begin{equation}\label{20221117_10:46}
(\mathbb{EP}1)^{-1}  \leq \int_{\R} |g(x)|^2\,|x|^{2 \beta +1}\,\d x.
\end{equation}
In particular $(\mathbb{EP}1)^{-1}$ is finite (as we had already observed with a different reasoning in the introduction). From the setup of our problem, we must have equalities both in \eqref{20221117_10:45} and \eqref{20221117_10:46}, and $g \in \mc{H}_{\alpha}(1\,; 1)$ is therefore an extremizer.

\subsection{Even extremizers in dimension $d=1$} Having established the existence of extremizers in dimension $d=1$ in \S \ref{Exist_extr_dim_1}, we now prove the following complementary result.

\begin{proposition}\label{Prop_even_extremizers_dim1}
Let $\alpha >\beta > -1$ and $\delta >0$ be real parameters. Any extremizer for $(\mathbb{EP}1)(\alpha, \beta\,; 1\,; \delta)$ must be an even function.
\end{proposition} 
\begin{proof}
Let $g \in \mc{H}_{\alpha}(1\,; \delta)$ be an extremizer of $(\mathbb{EP}1)(\alpha, \beta\,; 1\,; \delta)$. Write 
\begin{equation*}
g(z) = g_{e}(z) + g_{o}(z),
\end{equation*} 
where $g_{e}(z) := \tfrac12(g(z) + g(-z))$ and $g_{o}(z) := \tfrac12(g(z) - g(-z))$ are the even and odd parts of $g$. Due to the orthogonality between $g_{e}$ and $g_{o}$ observe that
\begin{align}\label{20221117_17:03}
(\mathbb{EP}1)(\alpha, \beta\,; 1\,; \delta) = \frac{\int_{\R} |g(x)|^2 \, |x|^{2\alpha + 1}\,\d x }{\int_{\R} |g(x)|^2 \, |x|^{2\beta + 1}\,\d x} = \frac{\int_{\R} |g_{e}(x)|^2 \, |x|^{2\alpha + 1}\,\d x + \int_{\R} |g_{o}(x)|^2 \, |x|^{2\alpha + 1}\,\d x }{\int_{\R} |g_{e}(x)|^2 \, |x|^{2\beta + 1}\,\d x + \int_{\R} |g_{o}(x)|^2 \, |x|^{2\beta + 1}\,\d x} .
\end{align}

Assume $g_{o}$ is not identically zero. Since $g_{e}$ and $g_{o}$ are both in $\mc{H}_{\alpha}(1\,; \delta)$, by \eqref{20221117_17:03} we observe that $g_{o}$ must also be an extremizer. We may write $g_{o}(z) = z \,h(z)$ with $h \in \mc{H}_{\alpha +1}(1; \delta) \subset \mc{H}_{\alpha}(1; \delta)$. Then, by the setup of our problem, 
\begin{align}\label{20221117_17:11}
\frac{\int_{\R} |h(x)|^2 \, |x|^{2\alpha + 3}\,\d x }{\int_{\R} |h(x)|^2 \, |x|^{2\beta + 3}\,\d x} = \frac{\int_{\R} |g_{o}(x)|^2 \, |x|^{2\alpha + 1}\,\d x }{\int_{\R} |g_{o}(x)|^2 \, |x|^{2\beta + 1}\,\d x}  = (\mathbb{EP}1)(\alpha, \beta\,; 1\,; \delta)  \leq \frac{\int_{\R} |h(x)|^2 \, |x|^{2\alpha + 1}\,\d x }{\int_{\R} |h(x)|^2 \, |x|^{2\beta + 1}\,\d x}.
\end{align}
Letting $\d\mu(x) = |h(x)|^2 \, |x|^{2\beta + 1}\,\d x$ and assuming without loss of generality that $\d\mu$ is normalized so that $\int_{\R} \d\mu(x) = 1$, the inequality in \eqref{20221117_17:11} can be rewritten as
\begin{align}\label{20221117_17:25}
\int_{\R}|x|^{2\gamma +2}\, \d\mu(x) \leq \left(\int_{\R}|x|^{2\gamma}\, \d\mu(x) \right) \left( \int_{\R}|x|^{2}\, \d\mu(x) \right)\,,
\end{align}
with $\gamma = \alpha - \beta >0$. On the other hand, by H\"{o}lder's inequality we have
\begin{align}
\int_{\R}|x|^{2\gamma}\, \d\mu(x) & \leq \left(\int_{\R}|x|^{2\gamma +2}\, \d\mu(x)\right)^{\frac{2\gamma}{2\gamma +2}}\left(\int_{\R} \d\mu(x)\right)^{\frac{2}{2\gamma +2}} \label{20221117_17:26}
\end{align}
and
\begin{align}
\int_{\R}|x|^{2}\, \d\mu(x) & \leq \left(\int_{\R}|x|^{2\gamma +2}\, \d\mu(x)\right)^{\frac{2}{2\gamma +2}}\left(\int_{\R} \d\mu(x)\right)^{\frac{2\gamma}{2\gamma +2}}, \label{20221117_17:27}
\end{align}
which multiply out to yield
\begin{align}\label{20221117_17:28}
\left(\int_{\R}|x|^{2\gamma}\, \d\mu(x) \right) \left(\int_{\R}|x|^{2}\, \d\mu(x)\right) \leq \int_{\R}|x|^{2\gamma +2}\, \d\mu(x).
\end{align}
We must then have equalities in \eqref{20221117_17:25},\eqref{20221117_17:26}, \eqref{20221117_17:27} and \eqref{20221117_17:28}. By the case of equality in H\"{o}lder's inequality, this implies that $|x|^{2\gamma +2}$ is constant in the support of $\d\mu$, which is a contradiction.

\smallskip

The conclusion is that our original extremizer $g$ must be an even function.
\end{proof}

\subsection{A brief detour: some basic monotonicity properties}
It is worth mentioning the following consequence of Proposition \ref{Prop_even_extremizers_dim1}. 

\begin{proposition}\label{Prop_basic_ineq_a+1}
For $\alpha > \beta > -1$ and $\delta >0$ real parameters, and $d \in \N$, we have
\begin{equation}\label{20221116_18:40}
(\mathbb{EP}1)(\alpha, \beta\,; d\,; \delta) < (\mathbb{EP}1)(\alpha+1, \beta+1\,; d\,; \delta).
\end{equation}
\end{proposition}
\begin{proof}
From Theorem \ref{Thm_dimension_shifts} it suffices to verify \eqref{20221116_18:40} for $d=1$. For any $0\neq f \in \mc{H}_{\alpha+1 }(1\,; \delta)$ we have $f(z) \,z\in \mc{H}_{\alpha}(1\,; \delta)$ and hence
$$(\mathbb{EP}1)(\alpha, \beta\,; 1\,; \delta) \leq \frac{\int_{\R} |f(x)\,x|^2 \, |x|^{2\alpha + 1}\,\d x }{\int_{\R} |f(x)\,x|^2 \, |x|^{2\beta + 1}\,\d x} =  \frac{\int_{\R} |f(x)|^2 \, |x|^{2\alpha + 3}\,\d x }{\int_{\R} |f(x)|^2 \, |x|^{2\beta + 3}\,\d x}.$$
This plainly leads us to the inequality
\begin{equation}\label{20221117_17:52}
(\mathbb{EP}1)(\alpha, \beta\,; 1\,; \delta) \leq (\mathbb{EP}1)(\alpha+1, \beta+1\,; 1\,; \delta).
\end{equation}
If we had equality in \eqref{20221117_17:52}, letting $f$ be an extremizer for $(\mathbb{EP}1)(\alpha+1, \beta+1\,; 1\,; \delta)$, which is even by Proposition \ref{Prop_even_extremizers_dim1}, the function $f(z) \,z$ would be an odd extremizer to $(\mathbb{EP}1)(\alpha, \beta\,; 1\,; \delta)$, a contradiction.
\end{proof}

In sympathy with our asymptotics of Theorem \ref{Thm5 - Asympt}, we remark that the following monotonicity property also holds. 

\begin{proposition}
For $\alpha > \beta > -1$ and $\delta >0$ real parameters, and $d \in \N$, the map 
$$\alpha \mapsto \frac{1}{(\alpha - \beta)}\, \log\big((\mathbb{EP}1)(\alpha, \beta\,; d\,; \delta)\big)$$ 
is non-decreasing. 
\end{proposition}
\begin{proof}
From Theorem \ref{Thm_dimension_shifts} it suffices to prove these claims for $d=1$. Let $\alpha_2 >\alpha_1 > \beta$ and $f \in \mc{H}_{\alpha_2}(1; \delta)$ be normalized so that $\int_{\R} |f(x)|^2 \, |x|^{2\beta +1}\,\d x = 1$. Write $\d \mu(x) =  |f(x)|^2 \, |x|^{2\beta +1}\,\d x$.
Applying H\"{o}lder's inequality with exponent $ p = (\alpha_2 - \beta)/(\alpha_1 - \beta)$ 
we get
\begin{equation*}
(\mathbb{EP}1)(\alpha_1, \beta\,; 1\,; \delta) \leq \int_{\R} |x|^{2(\alpha_1 - \beta)}\d \mu(x) \leq \left(\int_{\R} |x|^{2(\alpha_2 - \beta)}\d \mu(x)\right)^{1/p}.
\end{equation*}
This leads us to 
\begin{align*}
(\mathbb{EP}1)(\alpha_1, \beta\,; 1\,; \delta)^{1/(\alpha_1 - \beta)} \leq (\mathbb{EP}1)(\alpha_2, \beta\,; 1\,; \delta)^{1/(\alpha_2 - \beta)},
\end{align*}
as desired. 
\end{proof}

\subsection{Existence of extremizers in higher dimensions} \label{Exist_Extre_Subsc} Let $g$ be an extremizer for $(\mathbb{EP}1)(\alpha, \beta\,; 1\,; \delta)$. By Proposition \ref{Prop_even_extremizers_dim1}, $g$ is even. From Theorem \ref{Thm_dimension_shifts} one plainly verifies that the lift ${\mc L}_d(g)$ is a radial extremizer for $(\mathbb{EP}1)(\alpha, \beta\,; d\,; \delta)$. 

\section{Continuity: proof of Theorem \ref{Thm_continuity}}\label{Sec_cont_proof}

In light of Theorem \ref{Thm_dimension_shifts} and the dilation relation \eqref{20220816_15:18_1}, in order to prove Theorem \ref{Thm_continuity} it suffices to show that $(\alpha, \beta) \mapsto (\mathbb{EP}1)(\alpha, \beta\,; 1\,; 1)$ is continuous in range $\alpha \geq \beta >-1$. 

\subsection{Approximation by Schwartz functions} Let $\mc{S}(\R)$ be the class of Schwartz functions. We shall first prove the following lemma. 

\begin{lemma}\label{Lem21_aprox_Schwartz}
Let $\alpha \geq \beta >-1$. Then 

\begin{equation*}
(\mathbb{EP}1)(\alpha, \beta\,; 1\,; 1) = \inf_{ 0 \neq f \in  \mc{S}(\R) \cap \mc{H}_{\alpha}(1; 1) } \frac{\int_{\R} |f(x)|^2 \, |x|^{2\alpha +1}\,\d x}{\int_{\R} |f(x)|^2 \, |x|^{2\beta +1}\,\d x}.
\end{equation*}
\end{lemma}
In other words, when finding $(\mathbb{EP}1)(\alpha, \beta\,; 1\,; 1)$, one can take the infimum over the subclass of the entire functions $f$ of exponential type at most $1$ whose restriction to $\R$ belongs to $\mc{S}(\R)$ (which, by the classical Paley-Wiener theorem, is equivalent to saying that $f \in \mc{S}(\R)$ and ${\rm supp}\big(\widehat{f}\big) \subset \big[-\tfrac{1}{2\pi}, \tfrac{1}{2\pi}\big]$).

\begin{proof} Let $\varphi \in C^{\infty}(\R)$ be an even function such that ${\rm supp}(\varphi)  \subset [-1,1]$; $\varphi \geq 0$; $\widehat{\varphi} \geq 0$; and $\int_{\R} \varphi(y) \, \d y = 1$. As usual, for $\varepsilon >0$, set $\varphi_{\varepsilon}(y):= \tfrac{1}{\varepsilon} \varphi\big( \tfrac{y}{\varepsilon}\big)$. Then $\widehat{\varphi_{\varepsilon}}(x) = \widehat{\varphi}(\varepsilon x)$, and we see that $0 \leq \widehat{\varphi_{\varepsilon}}(x)  \leq \widehat{\varphi_{\varepsilon}}(0) =1$ and $\lim_{\varepsilon \to 0} \widehat{\varphi_{\varepsilon}}(x)  = 1$ for all $x \in \R$. 

\smallskip

Let $0 \neq f \in \mc{H}_{\alpha}(1; 1)$ be an extremizer for $(\mathbb{EP}1)(\alpha, \beta\,; 1\,; 1)$ (being an extremizer here is not really crucial, it would suffice to take a near-extremizer in this argument). Set 
$$f_{\varepsilon}(x) := f(x)  \,\widehat{\varphi_{\varepsilon}}(x).$$

\smallskip

We claim that $f_{\varepsilon} \in \mc{S}(\R)$. In fact, note that 
$f^{(n)}_{\varepsilon} =\sum_{k=0}^{n}  \binom{n}{k} f^{(k)} \, \widehat{\varphi_{\varepsilon}}^{(n-k)}$ and recall the fact that the de Branges space ${\mc H}(E_{\alpha})$ (equal to $\mc{H}_{\alpha}(1; 1)$ as sets, with norms differing by a multiplicative constant) is closed under differentiation; see \cite[Theorem 20]{CL1}. To verify our claim, it then suffices to prove that any $h \in {\mc H}(E_{\alpha})$ has at most polynomial growth. For this, recall the pointwise bound \eqref{20230221_15:39}
\begin{align*}
|h(x)| \leq \|h\|_{\mc{H}(E_{\alpha})}\, K_{\alpha}(x, x)^{1/2}\,,
\end{align*}
and the fact that $x \mapsto K_{\alpha}(x, x)$ is even. For $x >0$, using \eqref{Intro_Def_K}, \eqref{20221130_14:53} and \eqref{20230210_14:02}, we have
\begin{align*}
\pi K_{\alpha}(x,x) & = A_{\alpha}(x)^2 +  B_{\alpha}(x)^2 - \frac{(2\alpha +1)A_{\alpha}(x)B_{\alpha}(x) }{x}\\
& = \frac{2^{2\alpha} \,\Gamma(\alpha+1)^2}{x^{2\alpha +1}} \left(x\, J_{\alpha}(x)^2  + x\, J_{\alpha+1}(x)^2  - (2\alpha +1) J_{\alpha}(x) \, J_{\alpha+1}(x)\right).
\end{align*}
From the well-known asymptotic 
$$J_{\alpha}(x) = \left(\frac{2}{\pi x}\right)^{1/2} \cos\left( x - \frac{\alpha \pi}{2} - \frac{
\pi}{4}\right) + O_{\alpha}\big(x^{-3/2}\big)\,,$$
one gets that $K_{\alpha}(x,x) \sim x^{-2\alpha -1}$ for $x$ large.

\smallskip

We now observe that $\widehat{f_{\varepsilon}}  = \widehat{f}*\varphi_{\varepsilon}$. If $\alpha < -\frac12$, note that $\widehat{f}$ is in principle only a tempered distribution, but nevertheless $\widehat{f}*\varphi_{\varepsilon}$ is a function; see \cite[Theorem 3.13, Chapter I]{SW}. From the Paley-Wiener Theorem (one needs the version for distributions in \cite[Theorem 1.7.5 and Theorem 1.7.7]{Hor} if $\alpha < -\frac12$) we have ${\rm supp}\big(\widehat{f}\big) \subset \big[-\tfrac{1}{2\pi}, \tfrac{1}{2\pi}\big]$ and this leads us to ${\rm supp}\big(\widehat{f_{\varepsilon}}\big) \subset  \big[-\tfrac{1}{2\pi} - \varepsilon,\,  \tfrac{1}{2\pi} + \varepsilon \big]$. We now have just to dilate back by a factor that tends to $1$, by considering 
$$g^{\varepsilon}(x) := f_{\varepsilon}\left(\left(\frac{\tfrac{1}{2\pi} }{\tfrac{1}{2\pi}+ \varepsilon}\right)x\right).$$
Then ${\rm supp}\big(\widehat{g^{\varepsilon}}\big) \subset \big[-\tfrac{1}{2\pi}, \tfrac{1}{2\pi}\big]$. By dominated convergence (both in the numerator and denominator) note that 
$$\lim_{\varepsilon \to 0} \, \frac{\int_{\R} |g^{\varepsilon}(x)|^2 \, |x|^{2\alpha +1}\,\d x}{\int_{\R} |g^{\varepsilon}(x)|^2 \, |x|^{2\beta +1}\,\d x} = \frac{\int_{\R} |f(x)|^2 \, |x|^{2\alpha +1}\,\d x}{\int_{\R} |f(x)|^2 \, |x|^{2\beta +1}\,\d x}\,,$$
as desired.
\end{proof}

\subsection{Proof of Theorem \ref{Thm_continuity}} Let $\{\alpha_n\}_{n\geq 1}$ and $\{\beta_n\}_{n\geq 1}$ be sequences such that $\alpha_n \geq \beta_n >-1$ for all $n$, $\alpha_n \to \alpha$ and $\beta_n \to \beta$ as $n \to \infty$. We want to show that $(\mathbb{EP}1)(\alpha_n, \beta_n\,; 1\,; 1) \to (\mathbb{EP}1)(\alpha, \beta\,; 1\,; 1)$ as $n \to \infty$. The case $\alpha = \beta$ will follow from Theorem \ref{Thm5 - Asympt} and we exclude this case from the discussion below. Let us henceforth assume that $\alpha >\beta$.

\subsubsection{Upper bound} Given $\varepsilon >0$, by Lemma \ref{Lem21_aprox_Schwartz} there exists $\psi \in \mc{S}(\R)$ of exponential type at most $1$ with
$$\Phi_{\alpha, \beta}(\psi):= \frac{\int_{\R} |\psi(x)|^2 \, |x|^{2\alpha +1}\,\d x}{\int_{\R} |\psi(x)|^2 \, |x|^{2\beta +1}\,\d x} \leq (\mathbb{EP}1)(\alpha, \beta\,; 1\,; 1) + \varepsilon.$$
For such a $\psi$ we have $\Phi_{\alpha_n, \beta_n}(\psi) \to \Phi_{\alpha, \beta}(\psi)$ as $n\to \infty$ by the dominated convergence theorem. This plainly leads us to 
\begin{equation}\label{20230221_16:39}
\limsup_{n\to \infty}\ (\mathbb{EP}1)(\alpha_n, \beta_n\,; 1\,; 1) \leq (\mathbb{EP}1)(\alpha, \beta\,; 1\,; 1).
\end{equation}

\subsubsection{Lower bound} From Lemma \ref{Lem21_aprox_Schwartz}, let $\{\psi_n\}_{n\geq 1} \subset \mc{S}(\R)$ be a sequence of functions of exponential type at most $1$ such that 
\begin{equation}\label{20230222_11:05}
\Phi_{\alpha_n, \beta_n}(\psi_n) \leq (\mathbb{EP}1)(\alpha_n, \beta_n\,; 1\,; 1)  + \tfrac{1}{n}.
\end{equation}
Let us normalize $\psi_n$ so that 
\begin{equation}\label{20230221_15:26}
\int_{\R} |\psi_n(x)|^2 \,|x|^{2\beta+1} \, \d x = 1.
\end{equation}
The sequence $\{\psi_n\}_{n\geq 1}$ is then bounded in $\mc{H}_{\beta}(1; 1)$ and hence, passing to a subsequence if necessary, we may assume that $\psi_n \rightharpoonup \psi$ for some $\psi \in \mc{H}_{\beta}(1\,; 1)$. Since $\mc{H}_{\beta}(1; 1)$ is a reproducing kernel Hilbert space, we have already seen in \eqref{20230221_15:24} and \eqref{20230221_15:39} that this implies that $\psi_n \to \psi$ pointwise everywhere, and that $\psi_n$ is uniformly bounded in compact subsets of $\R$. By \eqref{20230221_15:26} and Fatou's lemma we get
\begin{equation}\label{20230221_15:27}
\int_{\R} |\psi(x)|^2 \,|x|^{2\beta+1} \, \d x \leq 1.
\end{equation}
We now state the following auxiliary result to this proof.

\begin{lemma}\label{Lem22_auxi_res_internal}
Let $\alpha >\beta$. Given any $\varepsilon>0$, there exists $R = R(\varepsilon)$ such that 
\begin{equation}\label{20230221_16:44}
\int_{[-R,R]^c} |\psi_n(x)|^2 \,|x|^{2\beta+1} \, \d x < \varepsilon\,,
\end{equation}
and
\begin{equation}\label{20230221_16:29}
\int_{[-R,R]^c} |\psi_n(x)|^2 \,|x|^{2\beta_n+1} \, \d x < \varepsilon\,,
\end{equation}
for $n \geq n_0(\varepsilon)$.
\end{lemma}  

Let us assume the validity of Lemma \ref{Lem22_auxi_res_internal} in order to conclude. From \eqref{20230221_15:26} and \eqref{20230221_16:44} we get 
\begin{equation*}
\int_{[-R,R]} |\psi_n(x)|^2 \,|x|^{2\beta+1} \, \d x > 1 - \varepsilon.
\end{equation*}
$n \geq n_0(\varepsilon)$, and by the dominated convergence theorem we have
\begin{equation*}
\int_{[-R,R]} |\psi(x)|^2 \,|x|^{2\beta+1} \, \d x \geq 1 - \varepsilon.
\end{equation*}
Since $\varepsilon >0$ was arbitrary, in light of \eqref{20230221_15:27} we conclude that 
\begin{equation}\label{20230221_15:46}
\int_{\R} |\psi(x)|^2 \,|x|^{2\beta+1} \, \d x = 1.
\end{equation}
Using \eqref{20230221_16:29}, \eqref{20230221_15:46}, dominated convergence and Fatou's lemma, one can also conclude that 
\begin{equation}\label{20230221_16:36}
\lim_{n \to \infty} \int_{\R} |\psi_n(x)|^2 \,|x|^{2\beta_n+1} \, \d x = 1.
\end{equation}

\smallskip

From \eqref{20230221_16:36}, Fatou's lemma and  \eqref{20230221_15:46} we conclude that  
\begin{align}\label{20230221_16:40}
\begin{split}
\liminf_{n \to \infty} \ (\mathbb{EP}1)(\alpha_n, \beta_n\,; 1\,; 1)  & = \liminf_{n \to \infty} \int_{\R} |\psi_n(x)|^2 \,|x|^{2\alpha_n+1} \, \d x \\
& \geq \int_{\R} |\psi(x)|^2 \,|x|^{2\alpha+1} \, \d x \geq (\mathbb{EP}1)(\alpha, \beta\,; 1\,; 1).
\end{split}
\end{align}
The desired continuity then follows from \eqref{20230221_16:39} and \eqref{20230221_16:40}.

\subsection{Proof of Lemma \ref{Lem22_auxi_res_internal}}
We first prove \eqref{20230221_16:29}. Assume the result is not true, i.e. that there exists $\varepsilon >0$ for which no $R$ works. Then, for any $R>1$, we have
\begin{equation}\label{20230222_09:28}
\limsup_{n \to \infty} \int_{[-R,R]^c} |\psi_n(x)|^2 \,|x|^{2\beta_n+1} \, \d x \geq \varepsilon.
\end{equation}
Since $\alpha >\beta$, we may take $\gamma = \tfrac{\alpha - \beta}{2}>0$ and assume without loss of generality that $\alpha_n - \beta_m \geq \gamma$ for any $n,m$. Therefore
\begin{align}\label{20230222_09:25}
\int_{[-R,R]^c} |\psi_n(x)|^2 \,|x|^{2\alpha_n+1} \,\d x \geq R^{2\gamma} \int_{[-R,R]^c} |\psi_n(x)|^2 \,|x|^{2\beta_n+1} \,\d x.
\end{align}
Since $\psi_n$ is uniformly bounded in compact subsets of $\R$, we have 
\begin{equation*}
 \left| \int_{[-R,R]} |\psi_n(x)|^2 \,|x|^{2\beta_n+1} \,\d x -  \int_{[-R,R]} |\psi_n(x)|^2 \,|x|^{2\beta+1} \,\d x\right| \leq \varepsilon
\end{equation*}
for $n$ large. In light of \eqref{20230221_15:26} this gives us
\begin{equation}\label{20230222_09:26}
\int_{[-R,R]} |\psi_n(x)|^2 \,|x|^{2\beta_n+1} \,\d x   \leq 1 + \varepsilon
\end{equation}
for $n$ large. From \eqref{20230221_16:39}, \eqref{20230222_11:05}, \eqref{20230222_09:25} and \eqref{20230222_09:26}, and then \eqref{20230222_09:28}, we have
\begin{align}\label{20230222_11:39}
\begin{split}
(\mathbb{EP}1)(\alpha, \beta\,; 1\,; 1) & \geq \limsup_{n\to \infty} \ (\mathbb{EP}1)(\alpha_n, \beta_n\,; 1\,; 1) = \limsup_{n\to \infty} \Phi_{\alpha_n, \beta_n}(\psi_n) \\
& \geq \limsup_{n\to \infty} \frac{R^{2\gamma} \int_{[-R,R]^c} |\psi_n(x)|^2 \,|x|^{2\beta_n+1} \,\d x}{(1 + \varepsilon) +\int_{[-R,R]^c} |\psi_n(x)|^2 \,|x|^{2\beta_n+1} \,\d x} \\
& \geq \frac{R^{2\gamma}}{\frac{1 + \varepsilon}{\varepsilon} + 1}.
\end{split}
\end{align}
This is a contradiction when $R$ is large.

\smallskip

The proof of \eqref{20230221_16:44} is in the same spirit. Assume it is not true, i.e. that there exists $\varepsilon >0$ such that for any $R>1$ we have
\begin{equation}\label{20230222_11:40}
\limsup_{n \to \infty} \int_{[-R,R]^c} |\psi_n(x)|^2 \,|x|^{2\beta+1} \, \d x \geq \varepsilon.
\end{equation}
In particular, let us also only consider $R > R(\varepsilon)$, with $R(\varepsilon)$ being such that \eqref{20230221_16:29} holds (that we already proved). In replacement of \eqref{20230222_09:25} we use
\begin{align}\label{20230222_11:41}
\int_{[-R,R]^c} |\psi_n(x)|^2 \,|x|^{2\alpha_n+1} \,\d x \geq R^{2\gamma} \int_{[-R,R]^c} |\psi_n(x)|^2 \,|x|^{2\beta+1} \,\d x.
\end{align}
The bound \eqref{20230222_09:26} continues to hold for $n$ large. Together with \eqref{20230221_16:29} this leads us to 
\begin{equation} \label{20230222_11:47}
\int_{\R} |\psi_n(x)|^2 \,|x|^{2\beta_n+1} \,\d x   \leq 1 + 2\varepsilon
\end{equation}
for $n$ large. Similarly to \eqref{20230222_11:39}, now using \eqref{20230222_11:40}, \eqref{20230222_11:41} and \eqref{20230222_11:47}, we obtain

\begin{align*}
\begin{split}
(\mathbb{EP}1)(\alpha, \beta\,; 1\,; 1) & \geq \limsup_{n\to \infty} \ (\mathbb{EP}1)(\alpha_n, \beta_n\,; 1\,; 1) = \limsup_{n\to \infty} \Phi_{\alpha_n, \beta_n}(\psi_n) \\
& \geq \limsup_{n\to \infty} \frac{R^{2\gamma} \int_{[-R,R]^c} |\psi_n(x)|^2 \,|x|^{2\beta+1} \,\d x}{1 + 2\varepsilon } \\
& \geq \frac{R^{2\gamma} \, \varepsilon}{1 + 2\varepsilon}\,,
\end{split}
\end{align*}
which is again a contradiction when $R$ is large. This concludes the proof of Lemma \ref{Lem22_auxi_res_internal}.

\section{Radial non-increasing delta majorant: proof of Theorem \ref{RBS=FU}}
Let $M \in \mc{R}^+(d\, ; 2\delta)$ and assume that $\int_{\R^d} M(\x)\,\d\x < \infty$. By Lemma \ref{type-to-ntype} we know that $M$ is the lift ${\mc L}_d(f)$ of a real entire function $f: \C \to \C$ of exponential type at most $2\delta$ that is also radial (i.e. even) non-increasing with
\begin{equation}\label{20221114_11:36}
 \int_{\R^d} M(\x)\,\d\x = \tfrac{1}{2}\,\omega_{d-1}  \int_{\R} f(x)\,|x|^{d-1}\,\d x. 
\end{equation}
The fact that $M({\bf 0}) \geq 1$ implies that $f(0) \geq 1$, and of course we may assume that $f(0) = 1$ (otherwise we could dilate down and get a smaller integral). Let $g = f'$. Since $f$ has exponential type at most $2\delta$ and belongs to $L^1(\R)$ by \eqref{20221114_11:36}, by a classical result of Plancherel and P\'{o}lya \cite{PP}, so does $g$. We may then write
\begin{align}\label{FTC}
f(x) = \int_{-\infty}^x g(u)\,\d u.
\end{align}

Using \eqref{FTC} and Fubini's theorem we arrive at
\begin{align}\label{Min_object}
\int_{\R} f(x)\,|x|^{d-1}\,\d x = \frac{1}{d} \int_{\R} |g(u)|\,|u|^d\,\d u.
\end{align}
Note that $H(u) = -u\, g(u)$ is a function of exponential type at most $2\delta$, that is non-negative on $\R$ and belongs to $L^1(\R)$. By Krein's decomposition in Lemma \ref{Krein_dec_lem} we may write $H(z) = z^2 \,h(z) \,h^*(z)$, where $h$ has exponential type at most $\delta$. This implies that
\begin{equation}\label{20221114_12:09}
g(z) = -z \,h(z) \,h^*(z).
\end{equation}
Since $g$ is odd, the condition that $f(0) = 1$ means that 
$$\int_{\R} |g(u)|\,\d u = 2\,,$$
which, in light of \eqref{20221114_12:09}, implies 
\begin{align}\label{Restatement_1}
\int_{\R} |h(u)|^2\,|u|\,\d u = 2.
\end{align}
From \eqref{20221114_11:36}, \eqref{Min_object} and \eqref{20221114_12:09} the quantity we want to minimize is
\begin{align*}
\frac{\omega_{d-1} }{2d}\int_{\R} |h(u)|^2\,|u|^{d+1}\,\d u\,,
\end{align*}
and we plainly conclude that 
$$(\mathbb{EP}2)(d\, ; \delta)\geq \frac{ \omega_{d-1}}{d} \, (\mathbb{EP}1)(\tfrac{d}{2}, 0\, ; 1\,; \delta).$$

Conversely, we have seen that $(\mathbb{EP}1)(\tfrac{d}{2}, 0\, ; 1\,; \delta)$ admits an extremizer $h \in {\mc H}_{\frac{d}{2}}(1\,; \delta)$ that is even. With the normalization \eqref{Restatement_1}, we can reverse the steps above by defining $g$ as in \eqref{20221114_12:09}, then $f$ as in \eqref{FTC}, and then $M = {\mc L}_d(f)  \in \mc{R}^+(d\, ; 2\delta)$ with the help of Lemma \ref{type-to-ntype}. This shows that
$$(\mathbb{EP}2)(d\, ; \delta)\leq \frac{\omega_{d-1} }{d} \, (\mathbb{EP}1)(\tfrac{d}{2}, 0\, ; 1\,; \delta).$$
We therefore conclude that 
$$(\mathbb{EP}2)(d\, ; \delta) =  \frac{\omega_{d-1}}{d} \, (\mathbb{EP}1)(\tfrac{d}{2}, 0\, ; 1\,; \delta)\,,$$
and that extremizers exist for $(\mathbb{EP}2)(d\, ; \delta)$. 

\section {Asymptotics: proof of Theorem \ref{Thm5 - Asympt}}\label{Sec5_asym_ppp}
Throughout this section we work in dimension $d=1$ and exponential type $\delta =1$. We may also assume that $\alpha > \beta$.
\subsection{Upper bound} \label{Upper bounds - example}Recall the functions $A_{\alpha}$ and $B_{\alpha}$ defined in \eqref{20220816_17:43} and \eqref{20221121_10:21}, respectively. Note that $A_{\alpha}(0)=1$ and $B_{\alpha}(0)=0$. Within the framework of \S \ref{Sub_DB_spaces}, let us consider the test function
$$f(z) = K_{\alpha}(0, z) = \frac{B_{\alpha}(z)}{\pi z}.$$ 
From the rightmost identity in \eqref{20221130_14:53}, as $z\to 0$, we get $B_{\alpha}'(0) = 1/(2\alpha+2)$. Then, by \eqref{Lem17_ii}, 
\begin{align*}
\int_{\R} |f(x)|^2 \, |x|^{2\alpha +1}\,\d x & = (c_{\alpha})^{-1} \, \|f\|_{\mc{H}(E_{\alpha})}^2 = (c_{\alpha})^{-1} \, K_{\alpha}(0, 0) = \frac{B'_{\alpha}(0)}{\pi \, c_{\alpha}} = \frac{2^{2\alpha +1} \, \Gamma(\alpha +1)^2}{\pi^2 (2\alpha +2)}.
\end{align*}
Also, from \cite[p. 683, \S 6.574, Eq. 2]{GR}
\begin{align*}
\int_{\R} |f(x)|^2 \, |x|^{2\beta +1}\,\d x & = \frac{2^{2\alpha +1} \, \Gamma(\alpha +1)^2}{\pi^2} \int_{0}^{\infty} |J_{\alpha +1}(x)|^2 \, |x|^{2\beta - 2\alpha -1}\,\d x\\
& = \frac{2^{2\alpha +1} \, \Gamma(\alpha +1)^2}{\pi^2} \frac{\Gamma(2\alpha - 2\beta +1)\, \Gamma(\beta +1)}{ 2^{2\alpha - 2\beta +1} \, \Gamma(\alpha - \beta+1)^2\, \Gamma(2\alpha - \beta +2)}. 
\end{align*}
Therefore we find
\begin{align*}
\frac{\int_{\R} |f(x)|^2 \, |x|^{2\alpha +1}\,\d x}{\int_{\R} |f(x)|^2 \, |x|^{2\beta +1}\,\d x} & = \frac{ 2^{2\alpha - 2\beta} \, \Gamma(\alpha - \beta+1)^2\, \Gamma(2\alpha - \beta +2)}{(\alpha +1)\, \Gamma(2\alpha - 2\beta +1)\, \Gamma(\beta +1)}\\
& = 2^{2\alpha - 2\beta} \, \left(\frac{\beta +1}{\alpha +1}\right)\, 
\frac{ \Gamma(\alpha - \beta+1)^2\, \Gamma(2\alpha - \beta +2)}{\Gamma(2\alpha - 2\beta +1)\, \Gamma(\beta +2)}\,,
\end{align*}
and hence
\begin{align*}
(\mathbb{EP}1)(\alpha, \beta\,; 1\,; 1) \leq 2^{2\alpha - 2\beta} \, \left(\frac{\beta +1}{\alpha +1}\right)\, 
\frac{ \Gamma(\alpha - \beta+1)^2\, \Gamma(2\alpha - \beta +2)}{\Gamma(2\alpha - 2\beta +1)\, \Gamma(\beta +2)}.
\end{align*}
This leads us to
\begin{align}\label{20221213_10:05}
\log (\mathbb{EP}1)(\alpha, \beta\,; 1\,; 1) \leq 2(\alpha - \beta) \log 2 \,+ \log \left(\frac{\beta +1}{\alpha +1}\right)\, +\log\left(\frac{ \Gamma(\alpha - \beta+1)^2\, \Gamma(2\alpha - \beta +2)}{\Gamma(2\alpha - 2\beta +1)\, \Gamma(\beta +2)}\right).
\end{align}

Let us denote momentarily by $U_0(\alpha, \beta)$ the function appearing on the right-hand side of \eqref{20221213_10:05}. We now verify the following proposition.
\begin{proposition}
For $\alpha > \beta > -1$ we have
\begin{equation}\label{20221213_12:12}
U_0(\alpha, \beta) = 2(\alpha - \beta)\log(\alpha +2) + \log \left(\frac{\beta +1}{\alpha +1}\right) + O(\alpha - \beta)\,,
\end{equation}
where the implied constant is universal.
\end{proposition}
\begin{proof}
In order to verify \eqref{20221213_12:12}, if suffices to verify that 
\begin{equation*}
U_1(\alpha, \beta) :=\log\left(\frac{ \Gamma(\alpha - \beta+1)^2\, \Gamma(2\alpha - \beta +2)}{\Gamma(2\alpha - 2\beta +1)\, \Gamma(\beta +2)}\right) - 2(\alpha - \beta)\log(\alpha +2)  = O(\alpha - \beta).
\end{equation*}
Note that 
\begin{align*}
U_1(\alpha, \beta)  = 2 \log  \Gamma(\alpha - \beta+1) & + \log \Gamma(2\alpha - \beta +2) \\
& - \log \Gamma(2\alpha - 2\beta +1) - \log \Gamma(\beta +2) - 2(\alpha - \beta)\log(\alpha +2) .
\end{align*}
Letting $x = \alpha - \beta$ and $y = \alpha + 1$, we note that $y > x > 0$ and rewrite $U_1(\alpha, \beta) = U_2(x,y)$ with 
\begin{align*}
U_2(x, y)  = 2 \log  \Gamma(x+1) + \log \Gamma(x + y +1) - \log \Gamma(2x +1) - \log \Gamma(y - x +1) - 2x\log(y +1).
\end{align*}
We want to show that $U_2(x, y)  = O(x)$, and towards this goal we split our analysis into two cases. Note that $U_2(0, y)=0$.  

\smallskip

\noindent {\it Case 1: $0 < x \leq 1$}. We differentiate with respect to $x$ and show that this partial derivative is uniformly bounded. The claim then follows by the mean value inequality. We have
\begin{equation*}
\frac{\partial}{\partial x} U_2(x, y)  = 2 \frac{\Gamma'}{\Gamma} (x+1) +  \frac{\Gamma'}{\Gamma} (x+y +1) - 2 \frac{\Gamma'}{\Gamma} (2x +1) +  \frac{\Gamma'}{\Gamma}(y-x+1) - 2\log(y+1).
\end{equation*}
We now use Stirling's formula
$$\frac{\Gamma'}{\Gamma}(s) = \log s + O(|s|^{-1})\,,$$
that holds uniformly in the range $-\pi + \varepsilon < {\rm arg}(s) < \pi - \varepsilon$, for any fixed $\varepsilon >0$ (which is the case; we are always in the real line and at least $1$). Hence
\begin{align*}
\frac{\partial}{\partial x} U_2(x, y)  &= 2 \log (x+1) +  \log (x+y +1) - 2\log (2x +1) + \log (y-x+1) - 2\log(y+1) + O(1)\\
& = \log \left(1 + \frac{x}{y+1}\right) +  \log \left(1 - \frac{x}{y+1}\right) + 2 \log (x+1)- 2\log (2x +1) + O(1)\\
& = O\left( \frac{x}{y+1}\right)  + O(x) + O(1) \\
& = O(1)\,,
\end{align*}
for $0 < x \leq 1$, independently of $y >x$. 

\smallskip

\noindent {\it Case 2: $x> 1$}. In this case, we use Stirling's formula in the following form
$$\log \Gamma(s) = (s - \tfrac12)\log s - s + O(1),$$
that also holds uniformly in the range $-\pi + \varepsilon < {\rm arg}(s) < \pi - \varepsilon$, for any fixed $\varepsilon >0$. We get
\begin{align*}
U_2(x, y)  & = (2x+1)\log(x+1) - (2x+2)  + (x + y + \tfrac12)\log(x+y+1) - (x+y+1) \\
& \ \ \ \ \  - (2x + \tfrac12)\log(2x+1)  + (2x+1) - (y-x+\tfrac12)\log(y - x+1) + (y-x+1) \\
& \ \ \ \ \  - 2x\log(y+1) + O(1)\\
& = (2x+1)\log(x+1)+ (x + y + \tfrac12)\log(x+y+1)\\
& \ \ \ \ \  - (2x + \tfrac12)\log(2x+1) - (y-x+\tfrac12)\log(y - x+1)  - 2x\log(y+1) + O(x)\\
& = -(2x+1) \log\left(2 - \frac{1}{x+1}\right) +\tfrac12 \log(2x+1) + (x + y + \tfrac12)\log(x+y+1)\\
&  \ \ \ \ \ - (y-x+\tfrac12)\log(y - x+1) - 2x\log(y+1) + O(x)\\
& = (x + y + \tfrac12)\log(x+y+1)- (y-x+\tfrac12)\log(y - x+1) - 2x\log(y+1) + O(x).
\end{align*}
If $x \leq (y+1)/2$ we rewrite this last expression as
\begin{align*}
&=  -(y + \tfrac12) \log \left( 1 - \frac{2x}{x + y + 1}\right) + x \log \left( 1 + \frac{x}{y + 1}\right) + x \log \left( 1 - \frac{x}{y + 1}\right)  + O(x) \\
 & = (y + \tfrac12) \, O\!\left( \frac{x}{y+1}\right) + x \,O\!\left( \frac{x}{y+1}\right) + O(x)\\
 & = O(x).
\end{align*}
If $x \geq (y+1)/2$, we rewrite it as
\begin{align*}
(x + y + 1)\log(x+y+1)- (y-x+1)\log(y - x+1) - 2x\log(y+1) + O(x).
\end{align*}
We then add and subtract $2x \log x$ to have
\begin{align*}
&= (x + y + 1)\log\left(\frac{y+x+1}{x}\right)- (y-x+1)\log\left(\frac{y-x+1}{x} \right) + O(x)\\
& = (x + y + 1)\log\left(\frac{y+x+1}{x}\right)- x\left(\frac{y-x+1}{x}\right)\log\left(\frac{y-x+1}{x} \right) + O(x).
\end{align*}
Noting that $a \mapsto a\log a$ is bounded for $0 < a \leq 1$ we conclude that the above is $O(x)$. This concludes the proof of the proposition.
\end{proof}

\subsection{Lower bound I: the case $\alpha > \beta \geq -\tfrac12$} \label{LB1_case1} As we saw in \S \ref{Upper bounds - example}, the upper bound is proved via a suitable example. The proof of the desired lower bound is more convoluted and we split our analysis into different scenarios; a major hurdle here is the singular range when $\alpha$ and $\beta$ are close to the endpoint $-1$, where the classical Fourier transform is not available.  

\smallskip

We first treat the situation where $\alpha > \beta \geq -\tfrac12$, where we can use uncertainty tools for the classical Fourier transform. In this range, our strategy is to use Lemma \ref{Lem9 - FU} and optimize our choice of parameters.

\subsubsection{Strategy} If $f \in \mc{H}_{\alpha}(1\,; 1)$, notice that $f \in L^2(\R)$ and it has a Fourier transform in the classical sense. Then, using Lemma \ref{Lem9 - FU} and the fact that $\beta \geq -\frac12$, we find 
\begin{align*}
\int_{\R} |f(x)|^2 & \,|x|^{2\beta +1} \,\d x  = \int_{[-s, s]} |f(x)|^2 \,|x|^{2\beta +1} \,\d x + \int_{[-s, s]^c} |f(x)|^2\, |x|^{2\beta +1} \,\d x\\
& \leq s^{2\beta +1} \int_{[-s, s]} |f(x)|^2 \,\d x + s^{2(\beta - \alpha)} \int_{[-s, s]^c} |f(x)|^2 \,|x|^{2\alpha +1} \,\d x\\
& \leq s^{2\beta +1}\,e^{Cr} \left(1 - e^{-Cs}\right) \int_{[-r,r]^c} |f(x)|^2\,\d x + s^{2(\beta - \alpha)} \int_{\R} |f(x)|^2\, |x|^{2\alpha +1} \,\d x\\
& \leq \frac{s^{2\beta +1}}{r^{2\alpha+1}}\,e^{Cr} \left(1 - e^{-Cs}\right) \int_{[-r,r]^c} |f(x)|^2\,|x|^{2\alpha +1}\, \d x + s^{2(\beta - \alpha)} \int_{\R} |f(x)|^2\, |x|^{2\alpha +1} \,\d x\\
& \leq \left(\frac{s^{2\beta +1}}{r^{2\alpha+1}}\,e^{Cr} \left(1 - e^{-Cs}\right)  + s^{2(\beta - \alpha)} \right) \int_{\R} |f(x)|^2\, |x|^{2\alpha +1} \,\d x.
\end{align*}
Hence
\begin{equation}\label{20230207_10:24}
(\mathbb{EP}1)(\alpha, \beta\,; 1\,; 1) \geq \left(\frac{s^{2\beta +1}}{r^{2\alpha+1}}\,e^{Cr} \left(1 - e^{-Cs}\right)  + s^{2(\beta - \alpha)} \right)^{-1}
\end{equation}
for any $r,s >0$. Our task is now to optimize the right-hand side. 

\subsubsection{Optimization and asymptotic analysis}  \label{Subsec_Opt_Asy_Ana}First we minimize a function of the type $r \mapsto e^{Cr}r^{-2\alpha -1}$. Taking logarithms we must minimize
$$g(r) := Cr - (2\alpha+1)\log r.$$
Such minimum occurs at $r_0 := (2\alpha +1)/C$. Plugging $r = r_0$ in our original expression \eqref{20230207_10:24}, we must then minimize the function 
$$h(s) := \frac{s^{2\beta +1}}{(2\alpha+1)^{2\alpha+1}}(Ce)^{2\alpha+1} \left(1 - e^{-Cs}\right)  + s^{2(\beta - \alpha)} 
.$$
This seems a bit harder to explicitly optimize in $s$. We make a simplification: using the fact that $1 - e^{-Cs}  \leq Cs$ for any $s>0$ we find that 
$$h_1(s) := \frac{C\, s^{2\beta +2}}{(2\alpha+1)^{2\alpha+1}}(Ce)^{2\alpha+1}   + s^{2(\beta - \alpha)}  \geq h(s).$$
From \eqref{20230207_10:24} note that $(\mathbb{EP}1)(\alpha, \beta\,; 1\,; 1) \geq h(s)^{-1} \geq h_1(s)^{-1}$, and we can minimize $h_1(s)$ instead. Letting $A :=  \frac{C}{(2\alpha+1)^{2\alpha+1}}(Ce)^{2\alpha+1}$, we have $h_1(s) = A s^{2\beta +2} +  s^{2(\beta - \alpha)}$\,, and the minimum occurs at 
$$s_0 = \left(\frac{\alpha - \beta}{A (\beta +1)}\right)^{1/(2\alpha +2)}.$$
We then get
\begin{align*}
h_1(s_0)^{-1} =  \frac{(\beta +1)^{\frac{\beta +1}{\alpha+1}}  (\alpha - \beta)^{\frac{\alpha - \beta}{\alpha+1}}}{A^{\frac{\alpha - \beta}{\alpha +1}} (\alpha +1)}\,,
\end{align*}
and 
\begin{align}\label{20230207_16:01}
\log \big((\mathbb{EP}1)& (\alpha, \beta\,; 1\,; 1)\big)  \geq \log h_1(s_0)^{-1} \nonumber \\
& = \frac{\beta +1}{\alpha+1} \log(\beta +1)  + \frac{\alpha - \beta}{\alpha+1} \log(\alpha - \beta) - \frac{\alpha - \beta}{\alpha+1}\log A - \log(\alpha +1)\nonumber\\
& = \log \left(\frac{\beta +1}{\alpha+1}\right) - \frac{\alpha - \beta}{\alpha+1} \log (\beta +1) + \frac{\alpha - \beta}{\alpha+1} \log(\alpha - \beta) -  2(\alpha - \beta)\log C \\
&  \ \ \ \ \ \ - \frac{(\alpha - \beta)(2\alpha +1)}{\alpha+1} + \frac{(\alpha - \beta)(2\alpha +1)}{\alpha+1} \log(2\alpha +1). \nonumber
\end{align}
Let us denote momentarily by $L_0(\alpha, \beta)$ the function appearing on the right-hand side of \eqref{20230207_16:01}. We now verify the following proposition.
\begin{proposition}
For $\alpha > \beta \geq -\frac12$ we have
\begin{equation*}
L_0(\alpha, \beta) = 2(\alpha - \beta)\log(\alpha +2)  + O\left((\alpha - \beta) \, \log \left( \frac {2(\alpha - \beta + 1)}{(\alpha - \beta)}\right)\right),
\end{equation*}
where the implied constant is universal.
\end{proposition}
\begin{proof}
If $\alpha \geq \beta \geq -\frac12$, we have $\log \left(\frac{\beta +1}{\alpha +1}\right) = O(\alpha - \beta)$. Hence 
\begin{align*}
L_0(\alpha, \beta) = 2(\alpha - \beta)\log(\alpha +2) + L_1(\alpha, \beta) + O(\alpha - \beta)\,,
\end{align*}
where $L_1(\alpha, \beta)$ is given by 
\begin{align*}
L_1(\alpha, \beta) :=  \frac{\alpha - \beta}{\alpha+1} \log(\alpha - \beta) + \frac{(\alpha - \beta)(2\alpha +1)}{\alpha+1} \log(2\alpha +1) - 2(\alpha - \beta)\log(\alpha +2).
\end{align*}
Letting $x = \alpha - \beta$ and $y = \alpha + 1$, we note that $y > x > 0$ and rewrite $L_1(\alpha, \beta) = L_2(x,y)$ with 
\begin{align*}
L_2(x,y) := \frac{x}{y}\log x  + \frac{x(2y -1)}{y}\log (2y -1) - 2x \log(y +1).
\end{align*}
We need to show that 
\begin{align*}
L_2(x,y) = O\left(x\log\left(\frac{2(x+1)}{x}\right)\right) .
\end{align*}

Since $y \geq \tfrac12$ we have $\frac{x}{y}\log x = O\big(x \log \big(2(x+1)/x\big)\big)$. The remaining part has the form $x \, L_3(y)$, with
$$L_3(y) := \frac{(2y -1)}{y}\log (2y -1) - 2 \log(y +1)\,,$$
and we claim that this quantity is uniformly bounded in $y$. In fact, for $\tfrac12 \leq y \leq 1$, this is the case (recall that $a\log a\to 0$ as $a \to 0^+$), and for $y \geq 1$ we rewrite it as
$$L_3(y) = 2 \log \left( 2 - \frac{3}{y+1}\right) - \frac{\log (2y - 1)}{y} = O(1).$$

\end{proof}

\subsection{Lower bounds II: the case $-1 < \beta < \alpha \leq 0$} \label{Sec_LBII} The exact choice of $0$ as the right endpoint for this case is not particularly important, as long as we are at a positive distance from $-1/2$; this will be suitable for \S \ref{Sec_LBIII}. In any case, to treat this singular range where both $\alpha$ and $\beta$ can be close to the lower endpoint $-1$, a different strategy is required. We use the reproducing kernel of the space.

\subsubsection{An alternative strategy} Let $f \in \mc{H}_{\alpha}(1\,; 1)$, and recall that $\mc{H}_{\alpha}(1\,; 1) = \mc{H}(E_{\alpha})$ as sets, with norms differing by a multiplicative constant (Lemma \ref{Sec5_rel_facts}). From the reproducing kernel identity
\begin{align}\label{20230310_14:19}
f(x) = \langle f, K_{\alpha}(x, \cdot) \rangle_{\mc{H}(E_{\alpha})} 
\end{align}
we get, after an application of the Cauchy-Schwarz inequality,
\begin{align}\label{20230208_16:46}
\begin{split}
|f(x)|^2 & \leq \|f\|_{\mc{H}(E_{\alpha})}^2 \langle K_{\alpha}(x, \cdot), K_{\alpha}(x, \cdot) \rangle_{\mc{H}(E_{\alpha})} = \|f\|_{\mc{H}(E_{\alpha})}^2\, K_{\alpha}(x, x)\\
& = c_\alpha \left(\int_{-\infty}^\infty |f(y)|^2 \,|y|^{2\alpha+1} \,\d y\right) K_{\alpha}(x, x).
\end{split}
\end{align}
From the pointwise bound \eqref{20230208_16:46} we obtain
\begin{align}\label{20230208_16:58}
\begin{split}
& \int_{\R} |f(x)|^2 \,|x|^{2\beta +1} \,\d x  = \int_{[-s, s]} |f(x)|^2 \,|x|^{2\beta +1} \,\d x + \int_{[-s, s]^c} |f(x)|^2\, |x|^{2\beta +1} \,\d x\\
& \leq c_\alpha \left(\int_{-\infty}^\infty |f(x)|^2 \,|x|^{2\alpha+1} \,\d x\right)  \int_{[-s, s]} K_{\alpha}(x, x) \, |x|^{2\beta +1} \,\d x + s^{2(\beta - \alpha)} \int_{[-s, s]^c} |f(x)|^2 \,|x|^{2\alpha +1} \,\d x\\
& \leq \left( c_\alpha \, \|K_{\alpha}(x, x)\|_{L^{\infty}[-s,s]} \, \frac{2 s^{2\beta +2}}{2\beta +2} + s^{2(\beta - \alpha)} \right) \left(\int_{-\infty}^\infty |f(x)|^2 \,|x|^{2\alpha+1} \,\d x\right)\,,
\end{split}
\end{align}
valid for any $s >0$. 

\smallskip

As long as we are inside the range
\begin{equation}\label{20230208_17:01}
s \leq 2 \sqrt{\alpha +1}\,,
\end{equation}
we can use Lemma \ref{Lem14_UBRK} to get 
\begin{align}\label{20230208_16:59}
\|K_{\alpha}(x, x)\|_{L^{\infty}[-s,s]} \leq \frac{1}{2 \pi (\alpha +1)} + \frac{s^2}{4\pi (\alpha +1)^2} \leq \frac{3}{2\pi(\alpha +1)}.
\end{align}
In this situation, from \eqref{20230208_16:58} and \eqref{20230208_16:59} we get
\begin{equation*}
(\mathbb{EP}1)(\alpha, \beta\,; 1\,; 1) \geq \left(\frac{\tfrac{3}{2\pi}\, c_\alpha}{(\alpha+1)(\beta+1)} \, s^{2\beta +2} + s^{2(\beta - \alpha)}\right)^{-1}
\end{equation*}
and we now seek to optimize the right-hand side in the admissible range \eqref{20230208_17:01}.

\subsubsection{Optimization and asymptotic analysis}  The global minimum of the function
$$h(s) := \frac{\tfrac{3}{2\pi}\, c_\alpha}{(\alpha+1)(\beta+1)} \, s^{2\beta +2} + s^{2(\beta - \alpha)}$$
for $s >0$ occurs at 
$$s_0 =  \left(\frac{(\alpha - \beta)(\alpha+1)}{\tfrac{3}{2\pi}\, c_{\alpha}}\right)^{\frac{1}{2\alpha +2}}.$$
First observe that such $s_0$ belongs to our admissible range \eqref{20230208_17:01}. In fact, recalling that $c_\alpha = \pi \,2^{-2\alpha-1}\, \Gamma(\alpha+1)^{-2}$, this claim is equivalent to the fact that 
\begin{equation*}
(\alpha - \beta) \leq \frac{3 (\alpha+1)^{\alpha}}{\Gamma(\alpha +1)^2} =  \frac{3(\alpha+1)^{\alpha+2}}{\Gamma(\alpha +2)^2}.
\end{equation*}
Since $(\alpha - \beta)  \leq (\alpha +1)$, it suffices to verify that 
$$1 \leq \frac{3 (\alpha+1)^{\alpha+1}}{\Gamma(\alpha +2)^2},$$
which is true in our range $-1 < \alpha \leq 0$ (note, however, that this would fail for large $\alpha$).

\smallskip

Proceeding as in \S \ref{Subsec_Opt_Asy_Ana} we have 
\begin{align*}
h(s_0)^{-1} =  \frac{(\beta +1)^{\frac{\beta +1}{\alpha+1}}  (\alpha - \beta)^{\frac{\alpha - \beta}{\alpha+1}}}{A^{\frac{\alpha - \beta}{\alpha +1}} (\alpha +1)}\,,
\end{align*}
with 
$$A = \frac{\tfrac{3}{2\pi}\, c_\alpha}{(\alpha+1)(\beta+1)} = \frac{3}{2^{2\alpha +2}\,\Gamma(\alpha+1)^2\,(\alpha+1)
\,(\beta+1)} = \frac{3 (\alpha+1)}{2^{2\alpha +2}\,\Gamma(\alpha+2)^2
\,(\beta+1)}.$$ 
This yields
\begin{align}\label{20230209_14:56}
\begin{split}
\log (\mathbb{EP}1)& (\alpha, \beta\,; 1\,; 1)  \geq \log h(s_0)^{-1} \\
& = \frac{\beta +1}{\alpha+1} \log(\beta +1)  + \frac{\alpha - \beta}{\alpha+1} \log(\alpha - \beta) - \log(\alpha +1) - \frac{\alpha - \beta}{\alpha+1}\log A\\
& = \log \left(\frac{\beta +1}{\alpha+1}\right) - \frac{\alpha - \beta}{\alpha+1} \log (\beta +1) + \frac{\alpha - \beta}{\alpha+1} \log(\alpha - \beta)  - \frac{\alpha - \beta}{\alpha+1} \log 3\\
& \ \ \ \ \ \ + 2(\alpha - \beta)\log 2 - \frac{\alpha - \beta}{\alpha+1} \log \left(\frac{\alpha +1}{\beta+1}\right)+ \frac{2(\alpha - \beta)}{\alpha+1} \log \Gamma(\alpha +2)\\
& = \log \left(\frac{\beta +1}{\alpha+1}\right) - \frac{\alpha - \beta}{\alpha+1}\log \left(\frac{3(\alpha +1)}{\alpha - \beta}\right) + 2(\alpha - \beta)\log 2  + \frac{2(\alpha - \beta)}{\alpha+1} \log \Gamma(\alpha +2).
\end{split} 
\end{align}
Let us denote momentarily by $L_4(\alpha, \beta)$ the function appearing on the right-hand side of \eqref{20230209_14:56}. We now verify the following proposition.
\begin{proposition}
For $-1 < \beta < \alpha \leq 0$ we have
\begin{equation*}
L_4(\alpha, \beta) = \log \left(\frac{\beta +1}{\alpha+1}\right) + O\left(\left(\frac{\alpha - \beta}{\alpha +1} \right) \log \left(\frac{2(\alpha +1)}{\alpha - \beta}\right)\right),
\end{equation*}
where the implied constant is universal.
\end{proposition}

\begin{proof} The last line of \eqref{20230209_14:56} is written in such a way that we are essentially done. Just observe that in our range we have $\log \Gamma(\alpha +2) = O(1)$ and 
$$\frac{\alpha - \beta}{\alpha+1} = O\left(\left(\frac{\alpha - \beta}{\alpha +1} \right) \log \left(\frac{2(\alpha +1)}{\alpha - \beta}\right)\right).
$$
\end{proof}

\subsection{Lower bounds III: the case $-1 < \beta < -\frac12$ and  $0 < \alpha$} \label{Sec_LBIII} In this remaining case, note that $\alpha - \beta \geq \frac12$. We argue by introducing a test point, say $-\frac14$ and using the previous two cases. In fact, from the definition of our problem we have the inequality  
\begin{equation}\label{20230209_22:04}
(\mathbb{EP}1)(\alpha, \beta\,; 1\,; 1) \geq (\mathbb{EP}1)(\alpha, -\tfrac14\,; 1\,; 1)\ (\mathbb{EP}1)(-\tfrac14, \beta\,; 1\,; 1).
\end{equation}
Hence, from \eqref{20230209_22:04} and our work in \S \ref{LB1_case1} and \S \ref{Sec_LBII}, we obtain
\begin{align*}
\log \big((\mathbb{EP}1)&(\alpha, \beta\,; 1\,; 1) \big)\geq \log \big((\mathbb{EP}1)(\alpha, -\tfrac14\,; 1\,; 1)\big) + \log \big((\mathbb{EP}1)(-\tfrac14, \beta\,; 1\,; 1)\big)\\
& \geq \Big(2(\alpha +\tfrac14)\log(\alpha +2)  + O(\alpha +\tfrac14)\Big) + \left(\log \left(\frac{\beta +1}{\frac34}\right) + O(1)\right)\\
& = 2(\alpha -\beta)\log(\alpha +2) + \log \left(\frac{\beta +1}{\alpha+1}\right) + O(\alpha - \beta),
\end{align*}
as desired.

\subsection{Conclusion} The combination of our work in \S \ref{Upper bounds - example}, \S \ref{LB1_case1},  \S \ref{Sec_LBII} and \S \ref{Sec_LBIII} leads to the unified formulation proposed in Theorem \ref{Thm5 - Asympt}.

\section{The sharp constant saga: multiplication by $z^k$ in de Branges spaces}\label{Sec8_mult_z}
\subsection{An extremal problem for general de Branges spaces} Let $E(z) = A(z) - iB(z)$ be a Hermite-Biehler function, with $A$ and $B$ real entire functions, and $\H(E)$ be the associated de Branges space as discussed in \S\ref{GEn_Ov_dB}. Throughout this section we assume that our function $E$ verifies the following conditions:
\begin{itemize}
\item[(C1)] $E$ has no real zeros.
\item[(C2)] The function $z \mapsto E(iz)$ is real entire (or, equivalently, $A$ is even and $B$ is odd).
\item[(C3)] $A \notin \H(E).$
\end{itemize}
We consider the following extremal problem.

\smallskip

\noindent {\it Extremal Problem 4} (EP4): Let $E = A - iB$ be a Hermite-Biehler function verifying the conditions (C1) - (C3) above, and let $k \in \N$. Assume, in addition, that $A$ has at least $k+1$ zeros. Find
\begin{align*}
(\mathbb{EP}4)(E; k) := \inf_{0 \neq f \in \mathcal{H}(E)} \frac{\|z^k \, f\|^2_{\mc{H}(E)}}{\|f\|^2_{\mc{H}(E)}}.
\end{align*}

\smallskip

Recall that the Hermite-Biehler condition \eqref{20230315_15:01} implies that all the zeros of $A$ and $B$ are real. Since $E$ has no real zeros, the sets of zeros of $A$ and $B$ are disjoint; in particular, $A(0) \neq 0$. Moreover, the zeros of $A$ and $B$ are all simple. In fact, if $A$ had a double zero $\xi$, then by \eqref{Intro_Def_K} one would have $K(\xi, \xi) = \|K(\xi, \cdot)\|^2_{\mc{H}(E)} = 0$, and then $K(\xi, \cdot)$ would be identically zero and \eqref{20210913_13:57} would imply that $B(\xi) = 0$, a contradiction. Condition (C3), that $A \notin \H(E)$, is generic and allows for the use of the interpolation formulas in \eqref{20210809_11:01}. The minor technical assumption that $A$ has at least $k+1$ zeros guarantees that the subspace 
\begin{align*}
\mc{X}_k(E)  := \{ f \in \mc{H}(E);\,  z^k f \in \mc{H}(E)\}
\end{align*}
is non-empty. In fact, if $\xi_1, \dots, \xi_{k+1}$ are zeros of $A$, the functions $K(\xi, z) = A(z) / \pi (z - \xi_j)$ all belong to $\mc{H}(E)$ and a suitable linear combination of them will decay like $|A(x)/x^{k+1}|$ as $|x| \to \infty$, and hence will be in $\mc{X}_k(E)$. 

\smallskip

We first establish a few qualitative properties about the extremal problem (EP4), in the spirit of similar results that have already been discussed in this paper (hence, we shall be brief in the proof). 
\begin{proposition} \label{Prop24_dB}The following statements hold:
\begin{itemize}
\item[(i)] $(\mathbb{EP}4)(E; k) >0$.
\item[(ii)] There exist extremizers for $(\mathbb{EP}4)(E; k)$.
\item[(iii)] Any extremizer for $(\mathbb{EP}4)(E; k)$ must be an even function. 
\end{itemize}
\end{proposition}
\begin{proof}
(i) Let $f \in \mc{X}_k(E)$ be such that $\|f\|^2_{\mc{H}(E)} = 1$. Recall, from the reproducing kernel identity and an application of the Cauchy-Schwarz inequality (like in \eqref{20230310_14:19} - \eqref{20230208_16:46}), that 
\begin{align}\label{20230310_15:00}
|f(z)|^2 & \leq \|f\|_{\mc{H}(E)}^2\, K(z, z)
\end{align}
for all $z \in \C$. Since $x \mapsto K(x, x)$ is continuous, we see that $f$ cannot concentrate at the origin. Letting $M = \max\{K(x,x); -1 \leq x \leq1\}$ and $\d \mu(x) = |E(x)|^{-2}\d x$, choose $\eta \leq 1$ such that $M \int_{[-\eta, \eta]}  \d\mu(x) \leq \frac{1}{2}$. Then $\int_{[-\eta, \eta]}  |f(x)|^2\, \d\mu(x) \leq \frac{1}{2}$ and 
\begin{align*}
\|z^k \, f\|^2_{\mc{H}(E)} \geq \int_{[-\eta,\eta]^c} |f(x)|^2 |x|^{2k}\,\d\mu(x) \geq \eta^{2k}\int_{[-\eta,\eta]^c} |f(x)|^2\,\d\mu(x) \geq \frac{\eta^{2k}}{2} >0.
\end{align*}

\smallskip

\noindent (ii) The proof is similar to \S \ref{Exist_extr_dim_1}. Let $\{f_n\}_{n\geq 1} \subset  \mc{X}_k(E)$ be an extremizing sequence, normalized so that $\|z^k f_n\|_{ \mc{H}(E)} = 1$. This means that $\|f_n\|^2_{ \mc{H}(E)} \to (\mathbb{EP}4)(E;k)^{-1}.$
By reflexivity, up to a subsequence, we may assume that $f_n \rightharpoonup g$ for some $g \in \mc{H}(E)$. The reproducing kernel identity implies that $f_n \to g$ pointwise everywhere, and by Fatou's lemma we have $\|z^k g\|_{ \mc{H}(E)} \leq 1$. In particular $g \in \mc{X}_k(E)$. Using \eqref{20230310_15:00}, note also that $f_n$ is uniformly bounded in compact subsets of $\C$. The proof that $\|g\|^2_{ \mc{H}(E)} \geq (\mathbb{EP}4)(E;k)^{-1}$ (and hence $g$ is the desired extremizer) follows the argument of \eqref{20221117_10:33}, \eqref{20221117_10:34} and \eqref{20221117_10:46}.

\smallskip

\noindent (iii) The proof follows the same reasoning of the proof of Proposition \ref{Prop_even_extremizers_dim1} (recall that $x \mapsto |E(x)|^{-2}$ is an even function in our case) and we omit the details.
\end{proof}

Our main result in this section is a complete solution for the extremal problem (EP4). Let 
$$0 < \xi_1 < \xi_2 < \xi_3 < \ldots$$
denote the sequence of positive zeros of $A$ and define the meromorphic function
$$C(z) := \frac{B(z)}{A(z)}.$$

\begin{theorem}[Sharp constants - general version]\label{Teoremacco_versao_dB}
 Let $k \in \N$ and set $ \lambda_0:=\big((\mathbb{EP}4)(E; k)\big)^{1/2k}$. 
 \begin{itemize}
 \item[(i)] If $k=1$ we have $\lambda_0 = \xi_1$. 
 \smallskip
 \item[(ii)] If $k \geq 2$, set $\ell := \lfloor k/2 \rfloor$. Then $\lambda_0$ is the smallest positive solution of the equation 
 \begin{align*}
 A(\lambda) \det \mc{V}(\lambda) = 0\,,
 \end{align*}
 where $\mc{V}(\lambda)$ is the $\ell \times \ell$ matrix with entries
\begin{align*}
\ \ \ \ \ \ \  \ \ \ \ \ \ \ \  \ \ \ \big(\mc{V}(\lambda)\big)_{mj} = \sum_{r=0}^{k-1} \omega^{r (4\ell - 2m - 2j +3)} \,C\big(\omega^{r} \lambda\big)  \ \ \  \ \  \ \ \ (1\leq m,j \leq \ell)\,;
\end{align*}
and $\omega := e^{\pi i/k}$.
 \end{itemize}
\end{theorem}

 \noindent {\sc Remark:} For $k \geq 2$, we shall verify in our proof that $0 < \lambda_0 < \xi_{\ell +1}$. Note that $\lambda \mapsto \det \mc{V}(\lambda)$ is, in principle, a meromorphic function of the variable $\lambda$ that is real-valued on $\R$, but we verify in our proof that the function $\lambda \mapsto A(\lambda) \det \mc{V}(\lambda)$ is in fact continuous on the interval $(0,\xi_{\ell +1})$.

\medskip

We are also able to classify the extremizers for $(\mathbb{EP}4)(E; k)$. In what follows, let 
\begin{equation}\label{20230317_18:08}
c_n:=\frac{-A'(\xi_n)}{B(\xi_n)}.
\end{equation}
From \eqref{Intro_Def_K} note that $c_n>0$. Also, for $\ell  = \lfloor k/2 \rfloor$ (when $k \geq 2$), let $\mc{T}$ be the $\ell \times \ell$ Vandermonde matrix with entries
\begin{equation}\label{20230317_18:13}
 \ \ \ \ \ \ \ \ \ \ \ \mc{T}_{mj} =\xi_{m}^{2\ell - 2j +1} \ \ \ \ \ \ (1\leq m,j \leq \ell)\,,
 \end{equation}
with $\mc{T}^{-1}$ denoting its inverse, and let $\mc{Q}(\lambda)$ be the $\ell \times \ell$ matrix with entries 
$$ \ \ \ \ \ \ \ \ \ \ \ \  \ \  \big(\mc{Q}(\lambda)\big)_{mj} = \frac{\big(\mc{V}(\lambda)\big)_{mj} }{(2k)\,\lambda^{2k - 4\ell + 2m +  2j -3}} \ \ \ \ \ \ (1\leq m,j \leq \ell).$$

\begin{theorem}[Classification of extremizers] \label{Thmaco_class}Let $k \in \N$ and $ \lambda_0=\big((\mathbb{EP}4)(E; k)\big)^{1/2k}$. 
\begin{itemize}
\item[(i)] If $k=1$, the extremizers for $(\mathbb{EP}4)(E; k)$ are spanned over $\C$ by the real entire function
\begin{equation*}
f(z) = \frac{A(z)}{(z^2 - \xi_1^2)}.
\end{equation*}
\item[(ii)] If $k \geq 2$, set $\ell := \lfloor k/2 \rfloor$. The extremizers for $(\mathbb{EP}4)(E; k)$ are spanned over $\C$ by the real entire functions
\begin{equation}\label{20230317_17:54}
f(z) = \sum_{n=1}^{\infty} a_n \frac{ \xi_n \,A(z)}{(z^2 - \xi_n^2)}\,,
\end{equation}
where 
\begin{align*}
a_n = \frac{ \sum_{i=1}^{\ell}c_ia_i \left(\sum_{r=1}^{\ell}  {\xi}_{n}^{2\ell - 2r +1} (\mc{T}^{-1})_{ri} \right) \big({\xi}_{i}^{2k} - \lambda_0^{2k}\big)}{ c_n\big( {\xi}_{n}^{2k} - \lambda_0^{2k}\big)} \ \ \ \ ;  \ \ \ \ {\rm for}\ n> \ell\,,
\end{align*}
and $(a_1, a_2, \ldots, a_\ell) \in \R^d\setminus\{{\bf 0}\}$ belongs to ${\rm ker}\,  \mc{W}(\lambda_0)$\footnote{Here we mean right multiplication, i.e. $(a_1, a_2, \ldots, a_\ell)\mc{W}(\lambda_0) = (0,0,\ldots,0)$.}, where $\mc{W}(\lambda_0)$ is the $\ell \times \ell$ matrix with entries 
\begin{align*}
\ \ \ \ \ \ \big(\mc{W}(\lambda_0)\big)_{ij} = c_i\big(\xi_i^{2k} - \lambda_0^{2k}\big) \left( \big(\mc{T}^{-1}\big)^t  \,  \mc{Q}(\lambda_0)\right)_{ij} \ \ \ \ \ \ (1\leq i,j \leq \ell).
\end{align*}
\end{itemize}
\end{theorem}

\noindent{\sc Remark:} If the even function $A$ has only finitely many zeros, say $2N$, the sum in \eqref{20230317_17:54} naturally goes only from $1$ to $N$. Throughout the proof below we slightly abuse the notation and keep writing the sums from $1$ to $\infty$, with this understanding. In addition, note that $\det \mc{W}(\lambda_0) = 0$ since $ A(\lambda_0) \det \mc{V}(\lambda_0) = 0$ and $\lambda_0 \in (0,\xi_{\ell+1})$; this will be further detailed in our proof. 

\smallskip

In the case $k=1$, Theorems \ref{Teoremacco_versao_dB} and \ref{Thmaco_class} have already been established in a previous work of the first author with Chirre and Milinovich \cite[Theorem 14]{CarChiMil}, with applications to number theory, namely, bounding the height of the first non-trivial zero in families of $L$-functions. We include this case here for completeness (and also since our argument below is slightly different from that of \cite{CarChiMil}).

\subsection{Proof of Theorem \ref{Teoremacco_versao_dB_Intro} and characterization of its extremizers} Note that Theorem \ref{Teoremacco_versao_dB_Intro} is a specialization of Theorem \ref{Teoremacco_versao_dB} in the case where $E(z) = E_{\beta}(z) = A_{\beta}(z) - iB_{\beta}(z)$ discussed in \S \ref{hom_dB_subsec}. If $\alpha = \beta +k$, note that $f \in  \mc{H}_{\alpha}(1; 1)$ if and only if $f \in \mc{X}_k(E_{\beta})$, and from Lemma \ref{Sec5_rel_facts} we plainly have 
$$(\mathbb{EP}1)(\beta+k, \beta\,; 1\,; 1) = (\mathbb{EP}4)(E_{\beta}; k).$$
The extremizers in Theorem \ref{Teoremacco_versao_dB_Intro} can be directly obtained from Theorem \ref{Thmaco_class} with $E = E_{\beta}$, and hence $\xi_n = \frak{j}_{\beta,n}$. In this case, note the additional simplification, due to \eqref{20221130_14:53}, that $c_n = 1$ for all $n\geq 1$.
\subsection{Preliminaries for the proofs of Theorems \ref{Teoremacco_versao_dB} and \ref{Thmaco_class}} Before diving into the proofs, we present a few useful remarks. 

\smallskip

We have seen in Proposition \ref{Prop24_dB} that any extremizer for $(\mathbb{EP}4)(E;k)$ must be an even function. If $f$ is an extremizer and we write $f(z) = g(z) - ih(z)$, where $g$ and $h$ are real entire functions (i.e. $g = (f + f^*)/2$ and $h = i(f - f^*)/2$), one has $|f(x)|^2 = |g(x)|^2 + |h(x)|^2$ for $x \in \R$ and, similarly to \eqref{20221117_17:03}, one readily sees that $g$ and $h$ must also be extremizers (if they are not identically zero). We shall see in our proof below that the set of real entire extremizers is a finite dimensional vector space over $\R$. Hence, the full space of extremizers is the span over $\C$ of the real entire extremizers.

\smallskip

Recall from \eqref{20210913_13:57} and \eqref{Intro_Def_K} that, for $A(\xi) = 0$, we have 
\begin{equation}\label{20230313_13:19_1}
K(\xi, z) = \frac{-A(z)B(\xi)}{\pi(z-\xi)} \ \ \ {\rm  and} \ \ \  K(\xi, \xi) =\frac{- A'(\xi)B(\xi)}{\pi}.
\end{equation}
If $f \in \mc{H}(E)$ is an even function, using \eqref{20230313_13:19_1}, we can group the zeros $+\xi$ and $-\xi$ in the interpolation formula \eqref{20210809_11:01} to get the representation
\begin{equation}\label{20230210_16:00_1}
f(z) = \sum_{n=1}^{\infty} \frac{2 \,\xi_n f(\xi_n)}{A'(\xi_n)} \frac{A(z)}{(z^2 - \xi_n^2)}.
\end{equation}
If $f \in \mc{H}(E)$ is odd, grouping the zeros $+\xi$ and $-\xi$ in the interpolation formula \eqref{20210809_11:01} we get
\begin{equation}\label{20230214_15:14_1}
f(z) = \sum_{n=1}^{\infty} \frac{2 f(\xi_n)}{A'(\xi_n)} \frac{z\, A(z)}{(z^2 - \xi_n^2)}. 
\end{equation}
Both representations \eqref{20230210_16:00_1} and \eqref{20230214_15:14_1} are uniformly convergent in compact subsets of $\C$.

\smallskip

If $f \in \mc{X}_k(E)$, then $g(z) := z^k f(z) \in \mc{H}(E)$ and we have the following constraints
\begin{equation}\label{20230210_16:04_1}
g(0) = g'(0) = \ldots = g^{(k-1)}(0) = 0.
\end{equation}
Reciprocally, if $g \in \mc{H}(E)$ is a function satisfying \eqref{20230210_16:04_1}, then $g(z) = z^k f(z)$ with $f \in \mc{X}_k(E)$.

\subsection{Proofs of Theorems \ref{Teoremacco_versao_dB} and \ref{Thmaco_class}: the case $k$ even} \label{SubSec_k_even_1}

\subsubsection{Sequential formulation} Let $k = 2\ell$ with $\ell \in \N$. Let $f \in \mc{X}_k(E)$ be an even and real entire function. In this case, since $g(z):=z^kf(z)$ is even, half of the constraints in \eqref{20230210_16:04_1} are already taken care of and we only have to worry about
\begin{equation}\label{20230214_15:17_1}
g(0) = g^{(2)}(0) = \ldots = g^{(2(\ell-1))}(0) = 0.
\end{equation}
From the representation \eqref{20230210_16:00_1} (for the function $g$) we have
\begin{align}\label{20230214_15:16_1}
g(z) = \sum_{n=1}^{\infty} \frac{2 \,\xi_n^{k+1} f(\xi_n)}{A'(\xi_n)} \frac{A(z)}{(z^2 - \xi_n^2)}\,,
\end{align}
and then
\begin{align}\label{20230214_15:18_1}
g(0) = 0 \iff \sum_{n=1}^{\infty} \frac{ \xi_n^{k-1} f(\xi_n)}{A'(\xi_n)} = 0.
\end{align}
Proceeding inductively by differentiating \eqref{20230214_15:16_1} and plugging in $z=0$ (the initial case being \eqref{20230214_15:18_1}), we come to the conclusion that \eqref{20230214_15:17_1} is equivalent to the set of identities\footnote{Note that in the rest of the proof we sometimes write our exponent with $2\ell$ instead of $k$. This is simply to facilitate the analogy for the upcoming proof when $k$ is odd in \S \ref{SubSec_k_odd_1}.}
\begin{align}\label{20230214_15:51_1}
\sum_{n=1}^{\infty} \frac{ \xi_n^{2\ell-2j +1} f(\xi_n)}{A'(\xi_n)} = 0 \ \ \ \ \ \ (j = 1, 2, \ldots ,\ell).
\end{align}

\smallskip

Let us set $a_n := f(\xi_{n})/A'(\xi_{n}).$
Note that $a_n$ is a real number since we are assuming that $f$ is real entire. Recall the definition of the positive quantity $c_n$ in \eqref{20230317_18:08}. From \eqref{20230313_13:19_1} and \eqref{20210809_11:01} we get 
\begin{align*}
\|f\|_{\mc{H}(E)}^2 = 2\pi \sum_{n=1}^{\infty} c_n a_n^2  \ \ \ \ \ {\rm and} \ \ \ \ \ \|z^k f(z)\|_{\mc{H}(E)}^2 = 2\pi \sum_{n=1}^{\infty} c_n a_n^2 \, {\xi}_{n}^{2k}.
\end{align*}
Letting $\lambda_0:=\big((\mathbb{EP}4)(E; k)\big)^{1/2k} $, an equivalent formulation of our problem is then 
$$\lambda_0^{2k}  = \inf_{\{a_n\} \in \A} \frac{\sum_{n=1}^{\infty} c_n a_n^2 \, {\xi}_{n}^{2k}}{\sum_{n=1}^{\infty} c_na_n^2 }\,,$$
where the infimum is taken over the class $\A$ of real-valued sequences $\{a_n\}_{n \geq 1}$, non-identically zero, that verify:  
\begin{itemize}
\item $\sum_{n=1}^{\infty} c_n a_n^2 \, \xi_n^{2k} < \infty$; 
\item and the $\ell$ constraints given by \eqref{20230214_15:51_1}, i.e. 
\begin{align}\label{20230214_16:22_1}
\sum_{n=1}^{\infty} a_n\, {\xi}_{n}^{2\ell - 2j +1} = 0 \ \ \ \ \ \ (j = 1, 2, \ldots, \ell).
\end{align}
\end{itemize}
\noindent {\sc Remark:} The function $K(0, z) = A(0)B(z)/(\pi z)$ belongs to $\mc{H}(E)$, and by \eqref{20210809_11:01} we have 
$$\sum_{n=1}^{\infty} \frac{1}{c_n\, \xi_n^2} < \infty.$$ From this and the condition $\sum_{n=1}^{\infty} c_n a_n^2 \, \xi_n^{2k} < \infty$, we have an alternative way to see that \eqref{20230214_16:22_1} is absolutely summable by applying Cauchy-Schwarz appropriately.

\subsubsection{A basic upper bound} \label{7.3.2_BUB_1} Recall the definition of the $\ell \times \ell$ Vandermonde matrix $\mc{T}$ in \eqref{20230317_18:13}. The system of equations \eqref{20230214_16:22_1} essentially says that we can express  the variables $a_1, a_2, \ldots, a_{\ell}$ in terms of the variables $\{a_n\}_{n > \ell}$. In fact, letting 
\begin{align}\label{20230214_16:54_1}
S_j := \sum_{n >\ell} a_n \, {\xi}_{n}^{2\ell - 2j +1} \ \ \ \ (j = 1,2, \ldots, \ell)\,,
\end{align}
from \eqref{20230214_16:22_1} we have 
\begin{align}\label{20230214_17:47_1}
(a_1, \ldots, a_{\ell})\,  \mc{T} = (-S_1, -S_2,\ldots, -S_{\ell})\,,
\end{align}
and hence
\begin{align}\label{20230214_16:42_1}
a_i = - \sum_{r=1}^{\ell} S_r\, (\mc{T}^{-1})_{ri} \ \ \ \ \ (i =1, 2, \ldots, \ell).
\end{align}

\smallskip

An easy choice of $\{a_n\}_{n > \ell}$ is $a_{\ell+1} = 1$ and $a_n = 0$ for $n \geq \ell +2$. Then $\{a_i\}_{i=1}^{\ell}$ are given by \eqref{20230214_16:42_1} (they are not all zero since $\mc{T}$ is invertible) and with this test sequence we find that 
$$\lambda_0^{2k} \leq \frac{\sum_{n=1}^{\ell+1} c_na_n^2 \, {\xi}_{n}^{2k}}{\sum_{n=1}^{\ell+1} c_na_n^2 } <{\xi}_{\ell +1}^{2k},$$
and hence 
$$\lambda_0 < \xi_{\ell +1}.$$

\subsubsection{Solution via variational methods} From \eqref{20230214_16:54_1} and \eqref{20230214_16:42_1}, letting 
\begin{equation*}
F\big(\{a_n\}_{n > \ell}\big) = \sum_{i=1}^{\ell}c_i\left(\sum_{r=1}^{\ell} \left(\sum_{n >\ell} a_n \, {\xi}_{n}^{2\ell - 2r +1} \right)\, (\mc{T}^{-1})_{ri}\right)^2{\xi}_{i}^{2k} +   \sum_{n> \ell} c_na_n^2 \, {\xi}_{n}^{2k}
\end{equation*}
and
\begin{equation*}
G\big(\{a_n\}_{n > \ell}\big) = \sum_{i=1}^{\ell}c_i\left(\sum_{r=1}^{\ell} \left(\sum_{n >\ell} a_n \, {\xi}_{n}^{2\ell - 2r +1} \right)\, (\mc{T}^{-1})_{ri}\right)^2 +   \sum_{n> \ell} c_na_n^2\,, 
\end{equation*}
our problem is to minimize $F/G$ over all the class $\A_1$ of real-valued sequences $\{a_n\}_{n > \ell}$, non-identically zero, that verify  $\sum_{n>\ell}c_n a_n^2 \, \xi_n^{2k} < \infty$.

\smallskip

From Proposition \ref{Prop24_dB} we know that there exists an extremal sequence, that we henceforth denote by $\{a_n^*\}_{n > \ell}$. This means that 
$$\lambda_0^{2k} = \frac{F\big(\{a_n^*\}_{n > \ell}\big) }{G\big(\{a_n^*\}_{n > \ell}\big) }\,,$$
and if we perturb this sequence by considering $\{a_n^* + \varepsilon b_n\}_{n > \ell}$, where $\{b_n\}_{n > \ell} \in \A_1$ and $\varepsilon$ is small, we get that the function 
$$\varphi(\varepsilon) := \frac{F\big(\{a_n^* + \varepsilon b_n\}_{n > \ell}\big) }{G\big(\{a_n^*+ \varepsilon b_n\}_{n > \ell}\big) }$$
is differentiable and must verify 
$$\varphi'(0) = 0.$$
In particular, with the perturbation $\{b_n\}_{n > \ell}$ given by $b_n = 1$ if $n = \ell + m$ and $b_n = 0$ otherwise (for any $m \in \N$), we arrive at the Lagrange multipliers 
\begin{equation}\label{20230214_17:26_1}
\frac{\partial F}{\partial a_n} \big(\{a_n^*\}_{n > \ell}\big)= \lambda_0^{2k} \, \frac{\partial G}{\partial a_n} \big(\{a_n^*\}_{n > \ell}\big)
\end{equation}
for all $n > \ell$. 

\smallskip

Using \eqref{20230214_16:42_1}, note that 
\begin{align*}
\frac{\partial F}{\partial a_n} \big(\{a_n\}_{n > \ell}\big) = \sum_{i=1}^{\ell}(-2c_ia_i) \left(\sum_{r=1}^{\ell}  {\xi}_{n}^{2\ell - 2r +1} (\mc{T}^{-1})_{ri} \right) {\xi}_{i}^{2k} +  2 c_n a_n \, {\xi}_{n}^{2k}
\end{align*}
and
\begin{align*}
\frac{\partial G}{\partial a_n} \big(\{a_n\}_{n > \ell}\big) = \sum_{i=1}^{\ell}(-2c_ia_i) \left(\sum_{r=1}^{\ell}  {\xi}_{n}^{2\ell - 2r +1} (\mc{T}^{-1})_{ri} \right)+  2 c_na_n.
\end{align*}
Dividing by $2$ we find that \eqref{20230214_17:26_1} is equivalent to 
\begin{align}\label{20230215_11:11_1}
c_na_n^* \big( {\xi}_{n}^{2k} - \lambda_0^{2k}\big) = \sum_{i=1}^{\ell}c_i a_i^* \left(\sum_{r=1}^{\ell}  {\xi}_{n}^{2\ell - 2r +1} (\mc{T}^{-1})_{ri} \right) \big({\xi}_{i}^{2k} - \lambda_0^{2k}\big) 
\end{align}
for all $n >\ell$. Since we have already established in \S \ref{7.3.2_BUB_1} that $\lambda_0 < {\xi}_{\ell +1}$ we find that 
\begin{align}\label{20230214_17:36_1}
a_n^* = \frac{ \sum_{i=1}^{\ell}c_ia_i^* \left(\sum_{r=1}^{\ell}  {\xi}_{n}^{2\ell - 2r +1} (\mc{T}^{-1})_{ri} \right) \big({\xi}_{i}^{2k} - \lambda_0^{2k}\big)}{ c_n\big( {\xi}_{n}^{2k} - \lambda_0^{2k}\big)}
\end{align}
for all $n >\ell$. In particular, note that $(a_1^*, a_2^*, \ldots, a_{\ell}^*) \neq {\bf 0}$.

\smallskip

At this point, if we multiply both sides of \eqref{20230214_17:36_1} by ${\xi}_{n}^{2\ell - 2j +1}$ and add over $n > \ell$, using \eqref{20230214_16:54_1} (call it $S_j^*$ the corresponding quantity for the sequence $\{a_n^*\}_{n > \ell}$) we get 
\begin{align*}
S_j^* =\sum_{i=1}^{\ell}c_ia_i^* \sum_{n> \ell}  \frac{\left(\sum_{r=1}^{\ell}  {\xi}_{n}^{2\ell - 2r +1} (\mc{T}^{-1})_{ri} \right) \big({\xi}_{i}^{2k} - \lambda_0^{2k}\big) \,{\xi}_{n}^{2\ell - 2j +1}}{ c_n\big( {\xi}_{n}^{2k} - \lambda_0^{2k}\big)}.
\end{align*}
Recall from \eqref{20230214_17:47_1} that  
\begin{equation*}
S_j^* = - \sum_{i=1}^{\ell}a_i^* \, \mc{T}_{ij} = - \sum_{i=1}^{\ell}a_i^*\,  \xi_{ i}^{2\ell - 2j +1}.
\end{equation*}
If we subtract these two equations we get 
\begin{align}\label{20230215_10:52_1}
0 = \sum_{i=1}^{\ell} a_i^* \left[ \sum_{n> \ell}  \frac{c_i \left(\sum_{r=1}^{\ell}  {\xi}_{n}^{2\ell - 2r +1} (\mc{T}^{-1})_{ri} \right) \big({\xi}_{i}^{2k} - \lambda_0^{2k}\big) \,{\xi}_{n}^{2\ell - 2j +1}}{ c_n\big( {\xi}_{n}^{2k} - \lambda_0^{2k}\big)} + \xi_{ i}^{2\ell - 2j +1}\right]
\end{align}
for each $j = 1, 2, \ldots, \ell$. 

\smallskip

Denoting by $\mc{W}(\lambda)$ the $\ell \times \ell$ matrix with entries
\begin{equation}\label{20230313_16:01}
\big(\mc{W}(\lambda)\big)_{ij} =  \left(\sum_{n > \ell} \frac{c_i\left(\sum_{r=1}^{\ell} \xi_{ n}^{4\ell - 2r - 2j +2}\, (\mc{T}^{-1})_{ri}\right)}{c_n \big(\xi_{n}^{2k} - \lambda^{2k}\big)}\right)\big(\xi_{i}^{2k} - \lambda^{2k}\big) +\xi_{ i}^{2\ell - 2j +1}\,,
\end{equation}
what we have seen from \eqref{20230215_10:52_1} is that we must have $\det \mc{W}(\lambda_0) = 0$ and the vector $(a_1^*, a_2^*, \ldots, a_{\ell}^*) \in \R^{\ell}\setminus\{{\bf 0}\}$ belongs to the $\ker \mc{W}(\lambda_0)$.

\smallskip

Conversely, if $0< \lambda_0 < {\xi}_{\ell +1}$ is such that $\det \mc{W}(\lambda_0) = 0$, and we let $\{{\bf 0}\} \neq (a_1^*, a_2^*, \ldots, a_{\ell}^*) \in \ker W(\lambda_0)$, we may define $a_n^*$ for $n >\ell$ by \eqref{20230214_17:36_1}. Then the constraints \eqref{20230214_16:22_1} are simply \eqref{20230215_10:52_1}, and one can check directly that $\sum_{n\geq 1}c_n(a_n^*)^2 \, \xi_n^{2k} < \infty$. Working backwards, \eqref{20230215_11:11_1} follows from \eqref{20230214_17:36_1}. Multiplying \eqref{20230215_11:11_1} by $a_n^*$ and summing over $n> \ell$ we get (note the use of \eqref{20230214_16:54_1} and \eqref{20230214_16:42_1} below)
\begin{align*}
\sum_{n > \ell} c_n(a_n^*)^2 \big( {\xi}_{n}^{2k} - \lambda_0^{2k}\big) & = \sum_{n > \ell}  a_n^* \,\sum_{i=1}^{\ell}c_i a_i^* \left(\sum_{r=1}^{\ell}  {\xi}_{n}^{2\ell - 2r +1} (\mc{T}^{-1})_{ri} \right) \big({\xi}_{i}^{2k} - \lambda_0^{2k}\big)\\
& = \sum_{i=1}^{\ell}c_i a_i^* \left(\sum_{r=1}^{\ell}  S_r^* \, (\mc{T}^{-1})_{ri} \right) \big({\xi}_{i}^{2k} - \lambda_0^{2k}\big)\\
&= - \sum_{i=1}^{\ell}c_i(a_i^*)^2  \big({\xi}_{i}^{2k} - \lambda_0^{2k}\big).
\end{align*}
The last expression can be rewritten as
\begin{align*}
 \frac{\sum_{n=1}^{\infty} c_n(a_n^*)^2 \, {\xi}_{n}^{2k}}{\sum_{n=1}^{\infty} c_n(a_n^*)^2 } = \lambda_0^{2k}.
\end{align*}

\smallskip

The conclusion is that the desired value $\lambda_0$ is the smallest positive solution of the equation $\det \mc{W}(\lambda) = 0$ and all the real-valued extremal sequences are given by \eqref{20230214_17:36_1} with $\{{\bf 0}\} \neq (a_1^*, a_2^*, \ldots, a_{\ell}^*) \in \ker \mc{W}(\lambda_0)$. 

\subsubsection{Simplifying the determinant} We now work a bit on the determinant of the matrix $\mc{W}(\lambda)$, given by \eqref{20230313_16:01}, in order to describe it in terms of the companion functions $A$ and $B$ that define the de Branges space $\mc{H}(E)$; this is better suited for computational purposes. Note that $\lambda \mapsto \det \mc{W}(\lambda)$ is a continuous function on the interval $(0, \xi_{\ell +1})$.

\smallskip

Observe first that 
\begin{align}
\big(\mc{W}(\lambda)\big)_{ij} & =c_i\big(\xi_i^{2k} - \lambda^{2k}\big) \left[ \left(\sum_{n > \ell} \frac{\left(\sum_{r=1}^{\ell} \xi_{ n}^{4\ell - 2r - 2j +2}\, (\mc{T}^{-1})_{ri}\right)}{c_n \big(\xi_{n}^{2k} - \lambda^{2k}\big)}\right) + \frac{\xi_{ i}^{2\ell - 2j +1}}{c_i \big(\xi_{i}^{2k} - \lambda^{2k}\big)}\right] \nonumber \\
& = c_i\big(\xi_i^{2k} - \lambda^{2k}\big) \sum_{n =1}^{\infty} \frac{\left(\sum_{r=1}^{\ell} \xi_{ n}^{4\ell - 2r - 2j +2}\, (\mc{T}^{-1})_{ri}\right)}{c_n \big(\xi_{n}^{2k} - \lambda^{2k}\big)} \label{20230317_22:57}\\
& =  c_i\big(\xi_i^{2k} - \lambda^{2k}\big) \left( \big(\mc{T}^{-1}\big)^t  \,  \mc{Q}(\lambda)\right)_{ij}\,,\nonumber
\end{align}
where $\mc{Q}(\lambda)$ is the $\ell \times \ell$ matrix with entries given by 
\begin{align*}
\big(\mc{Q}(\lambda)\big)_{mj} = \sum_{n=1}^{\infty} \frac{\xi_n^{4\ell - 2m - 2j +2}}{c_n\big(\xi_n^{2k} - \lambda^{2k}\big)} \ \ \ \ \ \ (1\leq m,j \leq \ell).
\end{align*}
Note that the second equality in \eqref{20230317_22:57} is a consequence of the identity $\sum_{r=1}^{\ell} \xi_{n}^{2\ell - 2r +1}(\mc{T}^{-1})_{ri} = \delta_{n i}$, for $1 \leq n \leq \ell$. From \eqref{20230317_22:57} we arrive at 
\begin{equation}\label{20230317_23:33}
\det \mc{W}(\lambda) = \left(\prod_{i=1}^\ell c_i\big(\xi_i^{2k} - \lambda^{2k}\big)\right) \big(\det \mc{T} \big)^{-1}\det \mc{Q}(\lambda)\,,
\end{equation}
and our task now is to understand this last determinant. The next lemma is helpful to rewrite each entry $\big(\mc{Q}(\lambda)\big)_{mj}$ in terms of the companion functions $A$ and $B$.

\begin{lemma}\label{Lem26}
Let $k \in \N$ and set $\zeta := e^{2\pi i/k}$. For $s \in \Z$ with $0 \leq s \leq k-1$ we have
\begin{align*}
\sum_{r=0}^{k-1} \frac{\zeta^{-rs}}{x - \zeta^r y} =  \frac{k \,x^{k - s -1}\,y^s}{x^k - y^k}\,,
\end{align*}
as rational functions of the variables $x$ and $y$. 
\end{lemma}
\begin{proof}
We argue briefly via series expansions. Observe that 
\begin{align*}
\sum_{r=0}^{k-1} \frac{\zeta^{-rs}}{x - \zeta^r y} & = \frac{1}{x} \sum_{r=0}^{k-1} \frac{\zeta^{-rs}}{ 1- \zeta^r \left(\tfrac{y}{x}\right)} = \frac{1}{x} \sum_{r=0}^{k-1} \zeta^{-rs} \sum_{m=0}^{\infty} \zeta^{rm}\left(\tfrac{y}{x}\right)^m = \frac{1}{x} \sum_{m=0}^{\infty}\left(\tfrac{y}{x}\right)^m\sum_{r=0}^{k-1} \zeta^{r(m-s)}\\
& = \frac{k}{x} \sum_{m\geq 0;\, m \equiv s \,({\rm mod}\, k )} \left(\tfrac{y}{x}\right)^m = \frac{k}{x} \frac{y^s}{x^s} \left(\frac{1}{1 - \left(\tfrac{y}{x}\right)^k}\right) = \frac{k \,x^{k - s -1}\,y^s}{x^k - y^k}.
\end{align*}
\end{proof} 
Recall that the even function $\pi K(0, z) = A(0)B(z)/z$ belongs to $\mc{H}(E)$. From \eqref{20230210_16:00_1} we get
\begin{align*}
\frac{B(z)}{2z} = \sum_{n=1}^{\infty} \frac{A(z)}{c_n (\xi_n^2 - z^2)}.
\end{align*}
Then, with $C(z) = B(z)/A(z)$, we have 
\begin{align}\label{20230313_22:07}
\frac{C(z)}{2z} = \sum_{n=1}^{\infty} \frac{1}{c_n (\xi_n^2 - z^2)}.
\end{align}
With $\zeta = e^{2\pi i/k}$, by Lemma \ref{Lem26} with $x = \xi_n^2$ and $y = \lambda^2$, and \eqref{20230313_22:07}, we have
\begin{align}
\big(\mc{Q}(\lambda)\big)_{mj} & = \sum_{n=1}^{\infty} \frac{\xi_n^{4\ell - 2m - 2j +2}}{c_n\big(\xi_n^{2k} - \lambda^{2k}\big)} \nonumber \\
& = \sum_{n=1}^{\infty} \frac{1}{k\,  \lambda^{2k - 4\ell + 2m +  2j -4}}\frac{k \, \xi_n^{4\ell - 2m - 2j +2}\, \lambda^{2k - 4\ell + 2m +  2j -4}}{c_n\big(\xi_n^{2k} - \lambda^{2k}\big)} \nonumber \\
& = \sum_{n=1}^{\infty} \frac{1}{k\,  \lambda^{2k - 4\ell + 2m +  2j -4}}\sum_{r=0}^{k-1} \frac{\zeta^{-r (k - 2\ell + m + j -2)}}{c_n\big(\xi_n^2 - \zeta^r \lambda^2\big)} \nonumber \\
& = \sum_{r=0}^{k-1}\frac{\zeta^{-r (k - 2\ell + m + j -2)}}{k\,  \lambda^{2k - 4\ell + 2m +  2j -4}}\sum_{n=1}^{\infty}\frac{1}{c_n\big(\xi_n^2 - \zeta^r \lambda^2\big)} \nonumber \\
& = \sum_{r=0}^{k-1}\frac{\zeta^{-r (k - 2\ell + m + j -2)}}{k\,  \lambda^{2k - 4\ell + 2m +  2j -4}} \frac{C\big(\zeta^{r/2} \lambda\big)}{2 \,\zeta^{r/2} \lambda} \nonumber\\
& = \frac{\big(\mc{V}(\lambda)\big)_{mj} }{(2k)\,\lambda^{2k - 4\ell + 2m +  2j -3}} \,,\label{20230313_22:16}
\end{align}
where $\mc{V}(\lambda)$ is the $\ell \times \ell$ matrix with entries 
\begin{align*}
\big(\mc{V}(\lambda)\big)_{mj} = \sum_{r=0}^{k-1} \omega^{r (4\ell - 2m - 2j +3)} C\big(\omega^{r} \lambda\big) \ \ \ \ \ \ (1\leq m,j \leq \ell)\,,
\end{align*}
and $\omega := e^{\pi i/k}$. Recall that we are only interested in the roots of our original determinant in the open interval $(0,\xi_{\ell +1})$. Since the power of $\lambda$ in the denominator of $\big(\mc{Q}(\lambda)\big)_{mj}$ in \eqref{20230313_22:16} depends only on $m+j$, when the determinant is multiplied out, there will be a common power of $\lambda$ that can be factored out. Also, the expression $\prod_{i=1}^\ell c_i\big(\xi_i^{2k} - \lambda^{2k}\big)$ appearing in \eqref{20230317_23:33} can be replaced by $A(\lambda)$ (that has the same zeros in the interval $(0,\xi_{\ell +1})$). We conclude that, in the range  $(0,\xi_{\ell +1})$, one has $\det \mc{W}(\lambda) = 0$ if and only if $A(\lambda) \det \mc{V}(\lambda) = 0$. This completes the proof in this case.

\subsection{Proof of Theorem \ref{Teoremacco_versao_dB}: the case $k$ odd} \label{SubSec_k_odd_1} Now let $k = 2\ell +1$ with $\ell \in \N \cup \{0\}$. Let $f \in \mc{X}_k(E)$ be an even and real entire function. In this case, note that $g(z):=z^kf(z)$ is odd.

\smallskip

If $k=1$, there are no non-trivial constraints in \eqref{20230210_16:04_1}, and we plainly get 
$$\lambda_0^{2}  = \inf_{\{a_n\} \in \A} \frac{\sum_{n=1}^{\infty} c_n a_n^2 \, {\xi}_{n}^{2}}{\sum_{n=1}^{\infty} c_na_n^2 } = \xi_1^2,$$
with the only extremal sequences being the ones with $a_1 \neq 0$ and $a_n = 0$ for $n \geq 2$. 

\smallskip

If $k \geq 3$ (i.e. $\ell \geq 1$), then more than half of the constraints in \eqref{20230210_16:04_1} are already taken care of and we only have to worry about
\begin{equation}\label{20230215_13:18_1}
g'(0) = g^{(3)}(0) = \ldots = g^{(2\ell-1)}(0) = 0.
\end{equation}
From the representation \eqref{20230214_15:14_1} (for the function $g$) we have
\begin{align}\label{20230215_13:19_1}
g(z) = \sum_{n=1}^{\infty} \frac{2 \, \xi_n^k  f(\xi_n)}{A'(\xi_n)} \frac{z\, A(z)}{(z^2 - \xi_n^2)}\,,
\end{align}
and then
\begin{align}\label{20230215_13:20_1}
g'(0) = 0 \iff \sum_{n=1}^{\infty} \frac{ \xi_n^{k-2} f(\xi_n)}{A'(\xi_n)} = 0.
\end{align}
Proceeding inductively by differentiating \eqref{20230215_13:19_1} and plugging in $z=0$ (the initial case being \eqref{20230215_13:20_1}), we come to the conclusion that \eqref{20230215_13:18_1} is equivalent to the set of identities
\begin{align*}
\sum_{n=1}^{\infty} \frac{ \xi_n^{2\ell-2j +1} f(\xi_n)}{A'(\xi_n)} = 0 \ \ \ \ \ (j = 1, 2, \ldots ,\ell)\,,
\end{align*}
just like we had in \eqref{20230214_15:51_1}. The rest of the proof is completely identical to what was done in \S \ref{SubSec_k_even_1} (and that is the reason we kept track of the variables $k$ and $\ell$ independently in the previous argument, so that it could be applied line by line to this case as well).

\section*{Acknowledgements}

We are thankful to Andrés Chirre, Dimitar Dimitrov, Friedrich Littmann, Jeffrey Vaaler and Don Zagier for helpful discussions on the themes related to this paper. C. G-R. was partially supported by the Spanish Ministry of Science, Innovation and Universities, grant PID2023-150984NB-I00, and by the Centro de An\'{a}lise Matem\'{a}tica, Geometria e Sistemas Din\^{a}micos (CAMGSD). C. G-R. was also supported by the Spanish State Research Agency, through the Severo Ochoa and Mar\'{i}a de Maeztu Program for Centers and Units of Excellence in R\&D (CEX2020-001084-M) and acknowledges CERCA Programme/Generalitat de Catalunya for institutional support. L.O. acknowledges support of the Office of Naval Research through grant GRANT14201749 (award number N629092412126). A.O. and M.S. were supported by PID2023-146646NB-I00 funded by MICIU/AEI/10.13039/501100011033 and by ESF+, the Basque Government through the BERC 2022-2025 program, and through BCAM Severo Ochoa accreditation CEX2021-001142-S / MICIN / AEI / 10.13039/501100011033. M.S. also acknowledges the support of the grant RYC2022-038226-I. L.O., S.O. and M.S. thank ICTP for the hospitality during part of the preparation of this work.

\end{document}